\RequirePackage{fix-cm}
\documentclass[twocolumn]{svjour3}          
\smartqed  
\usepackage{amsmath}
\usepackage{amssymb,times,dsfont,pgf}
\usepackage[english]{babel}
\usepackage{lineno}
\usepackage{graphicx}
\usepackage{color}
\usepackage{multicol,float}
\usepackage{tikz}
\usetikzlibrary{shapes,backgrounds}
\usepackage{indentfirst}
\usepackage{enumitem}
\usepackage{subfigure}
\usepackage{xspace,pgf,latexsym,mathtools}
\usepackage[b]{esvect}
\usepackage{verbatim}
\usepackage{hyphenat}
\usepackage{caption}
\usepackage{natbib}

\def\secondcircle{(0:3.5cm) circle (1.9cm)}
\def\thirdcircle{(-5:6.5cm) circle (0.7cm)}
\def\fourcircle{(0:4.5cm) circle (0.5cm)}

\begin{document}

\title{Fuzzy implication functions constructed from general overlap functions and fuzzy negations
}


\author{Jocivania Pinheiro \and Benjamin Bedregal \and Regivan H.N. Santiago \and Helida Santos$^\ast$ \and Gra\c{c}aliz P. Dimuro \and Humberto Bustince}


\institute{J. Pinheiro \at
             Departamento de Ci\^encias Naturais, Matem\'atica e Estat\'istica,  Universidade Federal Rural do Semi-\'Arido, Mossor\'o-RN, Brazil \\
              \email{vaniamat@ufersa.edu.br}           
           \and
          B. Bedregal, R. Santiago \at
         Departamento de Inform\'atica e Matem\'atica Aplicada, Universidade Federal do Rio Grande
         do Norte, Natal-RN, Brazil \\
          \email{ \{bedregal,regivan\}@dimap.ufrn.br}       \and 
         H. Santos$^\ast$, G. Dimuro \at
          Centro de Ci\^encias Computacionais,  Universidade Federal do Rio Grande, Rio Grande-RS, Brazil \\
          \email{\{helida,gracalizdimuro\}@furg.br} 
         \and
          H. Bustince \at
         Departamento de Estad\'istica, Inform\'atica y Matem\'aticas, Institute of Smart Cities, Universidad Publica de Navarra, Pamplona-NA, Spain \\
         \email{bustince@unavarra.es} }

\maketitle

\begin{abstract}
		
Fuzzy implication functions have been widely investigated, both in theoretical and practical fields.  The aim of this work is to continue previous works related to fuzzy implications constructed by means of  non necessarily associative  aggregation functions. In order to obtain a more general and flexible context, we extend the class of implications derived by fuzzy negations and  t-norms, replacing the latter by general overlap functions. We also investigate their properties, characterization and intersections with other classes of fuzzy implication functions.
\end{abstract}

\section{Introduction}
\label{intro}
It is undeniable the importance of constructing implication functions from classes of aggregation functions that extend the classic Boolean disjunction to the unit interval, given the numerous studies found in the literature. These functions have been investigated in many ways, including studies with a more theoretical point of view and the ones dealing with practical applications, as seen in \cite{Bac13,BJ08,BBBP13,BJ2015,Mas07}, or in fields like approximate reasoning, decision making, expert systems, image processing, and fuzzy mathematical morphology found in \cite{Bloch09,BFSBM13,CBS18,jay08,Pradera16,RBB13,SdB13,Yager04}. There exists a wide range of methods to construct fuzzy implication functions as there are many ways of obtaining aggregation functions. In \cite{Bedregal07}, it was introduced  a family of fuzzy implications constructed from a t-norm $T$ and a fuzzy negation $N$. Recently, \cite{PBSS17,Vania-Nafips18,PBSS18,Vania-Fuzz18}  revisited this class of implication functions, calling them as $(T,N)$-implications. In those works some properties were investigated, including the definition of fuzzy subsethood measures by means of these fuzzy implication functions.

Initially, the investigations mostly used t-norms and t-conorms as in \cite{DP84}, however  operators  different from them can be used to construct implication-like functions, namely, uninorms or semi-uninorms by \cite{LIU12,Xie2012209}, pseudo-t-norms in \cite{Liu2011783,WANG06}, (dual) copulas, quasi- (semi-) copulas given by \cite{carbonell10,Dolati13}, overlap and grouping functions studied in \cite{Dimuro2015,DBS14,Dimuro2017-ql,d-impl}.

Regarding the studies related to fuzzy implications constructed from weaker operators, we highlight the ones considering non-necessarily associative aggregation operators in the definition of fuzzy implications, like
overlap and grouping functions. These functions were proposed by \cite{Bus10a,BPMHF12} within the scope of classification problems in which the separation of classes could be unclear. Later, it was observed that whenever there exist more than two classes, it was more suitable to deal with an aggregation function which accepted more than two inputs. Then, $n$-Dimensional overlap functions were introduced in \cite{Gomez201657}. Subsequently, \cite{DeMiguel2019} generalized that concept to deal with problems containing more than two classes defining the general overlap functions.

Note that, in classical logic, one can define the implication connective  in distinct ways, meaning that if the truth tables are equal, then the operators are equivalent, \cite{Mendelson15}. However, when one generalizes those equivalences to the unit interval $[0,1]$, different classes of fuzzy implication functions are obtained. For example, when we generalize the $\vee$ operator and replace it by the grouping function $G$, the $\wedge$ operator by the overlap function $O$ and $\neg$  by a fuzzy negation $N$, we can mention $(G,N)$-implication functions, by  \cite{DBS14}, which generalizes the material implication used in Kleene algebra that can be defined according to the tautology:
	\begin{equation}\label{eq-taut3a}
		p \rightarrow q \equiv \neg p \vee q.
	\end{equation}

Later, in \cite{Dimuro2015}, $R_{o}$-implication functions were proposed. They are implications given by means of overlap functions inspired on the generalization of Boolean implications resulted as the residuum of the conjunction of Heyting algebra considered in the intuitionistic logic and defined according to the identity: $A' \cup B = (A - B)'=\bigcup \{C \in X \colon(A \cap C) \subseteq B \}$, where $X$ is a universe set and $A,B \subseteq X$. Moreover, in \cite{Dimuro2017-ql}, the  implication functions defined in the quantum logic framework, were also generalized using the following tautology: $p \rightarrow q \equiv \neg p \vee (p \wedge q)$, known as $QL$-implication functions. And there is also the study on $D$-implication functions, given in \cite{d-impl} (also known as Dishkant implication), derived from the following generalization: $p \rightarrow q \equiv q \vee (\neg p \wedge \neg q)$.

The natural sequence of the study of fuzzy implication functions derived by overlap and grouping functions should consider the tautology: 
\begin{equation}\label{eq-taut3}
p \rightarrow q \equiv \neg (p \wedge \neg q)
\end{equation}
which was recently generalized by t-norms, and called $(T,N)$ -implication functions. However, a more general and flexible context may be obtained if one considers generalized overlap functions instead of the standard overlap functions. 

The aim of this paper is to introduce a new  family of fuzzy implications generalizing Eq. (\ref{eq-taut3}) to $[0,1]$, entitled by $(\mathcal{GO},N)$-implications, where $\mathcal{GO}$ is the set of general overlap functions and $N$ is a fuzzy negation. 
We study the properties satisfied by such implication functions, providing a characterization and studying the intersections between them and the other families of implication functions endowed with overlap and grouping functions.

The paper is organized as follows. Section 2 includes some definitions and concepts necessary for the development of our work. In Sections 3 and 4 the main contributions concerning $(\mathcal{GO},N)$-implication functions and intersections between families of fuzzy implications are provided. Finally, in Section 5 we address the final remarks and a brief outline on future works.
 
\section{Preliminary concepts}
\subsection{Fuzzy Negations}

Fuzzy negations have been widely investigated and the main notions and properties related to them can be seen in \cite{BJ08,Bedregal08,DBS14,Klir95,Tri79}.

\begin{definition} \label{Negacao} \cite{DBS14} 
	A mapping $N: [0,1] \rightarrow [0,1]$ is said to be a \textbf{fuzzy negation} if
	\begin{enumerate}[labelindent=\parindent, leftmargin=*,label=\normalfont{(N\arabic*)}]
		\item $N$ is antitonic, i.e. $ N(x) \leq N(y) $ if $y \leq x $;\label{N1}
		\item $ N(0) = 1 $ and $ N(1) = 0 $. \label{N2}\\
		
		A fuzzy negation $N$ is \textbf{strict} if
		\item $ N$ is continuous and \label{N3}
		\item $ N(x) < N(y) $ whenever $y < x $. \label{N4}\\
		
		A fuzzy negation $N$ is \textbf{strong} if
		\item $ N(N(x)) = x $, for each $ x \in [0,1]$. \label{N5}\\
		
		A fuzzy negation $N$ is \textbf{crisp} if
		\item $ N(x) \in \{0,1\} $, for all $ x \in [0,1]$. \label{N6}\\
		
		A fuzzy negation $N$ is \textbf{frontier} if it satisfies:
		\item $N(x) \in \{0,1\}$ if and only if $x=0$ or $x=1$.\label{N7}
		
	\end{enumerate}
	
\end{definition}

The standard (or Zadeh) negation is: $N_Z(x)\hspace{-0.05cm}=\hspace{-0.05cm} 1-x$.

\begin{remark}\label{rem-N-crisp}
	By \cite{Dimuro2017-ql}, a fuzzy negation $N: [0, 1] \rightarrow [0,1]$ is crisp if and only if there exists
	$\alpha \in \left[ 0,1 \right) $ such that $N = N_\alpha$ or there exists $\alpha \in \left( 0,1 \right] $ such that $N = N^\alpha$, where
	\begin{eqnarray*}
		N_{\alpha}(x) = \begin{cases}
			0, & \text{ if }  x > \alpha \\
			1, & \text{ if }  x \leq \alpha
		\end{cases}
	\text{ and }
		N^{\alpha}(x) = \begin{cases}
			0, & \text{ if }  x \geq \alpha \\
			1, & \text{ if } x < \alpha
		\end{cases}.
	\end{eqnarray*}	
\end{remark}

As examples of crisp fuzzy negations, we have  the smallest fuzzy negation $N_{\bot}$ and the greatest fuzzy negation $N_{\top}$, which are given by $N_{\bot} = N_0$ and $N_{\top} = N^{1}$, respectively.

In our subsequent developments, the notion of $N$-duality is going to play a very relevant role.

\begin{definition} \label{def_Ndual}
	Let $N$ be a fuzzy negation and $f:[0,1]^{n} \rightarrow [0,1]$ be any fusion function. The $\mathbf{N}$\textbf{-dual function} of $f$, for all $x_1, \ldots, x_n \in [0,1]$, is given by the expression:
	\begin{eqnarray}\label{eq-f_N}
	f_{N}(x_1, \ldots, x_n) = N(f(N(x_1), \ldots, N(x_n))).
	\end{eqnarray}
\end{definition}

\subsection{From Aggregation functions to General Overlap Functions}

\begin{definition} \cite{Beliakov:Aggregations}
	An $n$-ary \textbf{aggregation function} is a mapping $A: [0,1]^{n} \rightarrow [0,1]$ satisfying the following properties:
	\begin{enumerate}[labelindent=\parindent,leftmargin=*,label=\normalfont{(A\arabic*)}]
		\item $A(0, 0, \ldots, 0) = 0$ and $A(1, 1, \ldots, 1) = 1$;
		\item  For each $i \in \{1,\ldots, n \}$, if $x_i \leq y$, then
		\begin{equation*}
			A(x_1,\ldots, x_{n})\hspace{-0.5ex} \leq \hspace{-0.5ex} A(x_1, \ldots,x_{i-1}, y, x_{i+1},\ldots, x_{n}).
		\end{equation*}		
	\end{enumerate}
\end{definition}

\begin{proposition}\cite[Corollary 3.8]{RBB13} \label{AN_aggregation}
	Let $A: [0,1]^{n} \rightarrow [0,1]$ be an aggregation function and $N$ be a fuzzy negation. The $N$-dual function of $A$, $A_{N}: [0,1]^{n} \rightarrow [0,1]$, is also an aggregation function.
\end{proposition}

\begin{definition} \cite{KMP00}\label{def_t-norm}
	An aggregation  function $T:[0,1]^2\rightarrow[0,1]$ is a t-norm if it satisfies the following conditions, for all $x,y,z \in [0,1]$:
	\begin{enumerate} [labelindent=\parindent, leftmargin=*,label=\normalfont{(T\arabic*)}]
		\item  $T(x,y)=T(y,x)$;
		\item   $T(x,T(y,z))=T(T(x,y),z)$;
		\item   $T(x,1)=x$.
	\end{enumerate}\end{definition}
	
	\begin{definition} \label{Overlap}\cite{Bus10a}
		A binary function $O:[0,1]^2 \rightarrow [0,1]$ is said to be an \textbf{overlap function} if it satisfies the following conditions, for all $x, y, z \in [0,1]$:
		
		\begin{enumerate} [labelindent=\parindent, leftmargin=*,label=\normalfont{(O\arabic*)}]
			\item $O(x,y) = O(y,x)$;\label{O1}
			\item $O(x,y) = 0$ if and only if $x = 0$ or $y = 0$;\label{O2}
			\item $O(x,y) = 1$ if and only if $x = y = 1$;\label{O3}
			\item if $ x\leq y$ then $O(x,z) \leq O(y, z)$;\label{O4}
			\item $O$ is continuous;\label{O5}
			
		\end{enumerate}
		
	\end{definition}
	
	\begin{remark} \label{Elem.Neutro_Over1}
		Note that whenever an overlap function has a neutral element, then, by \ref{O3}, it is necessarily  $1$.
	\end{remark}
	For further properties and related concepts on overlap functions, see \cite{BDBB13,BPMHF12,Dimuro201439,Dimuro2015,additive-generators-FSS,Dimuro2017-ql,Jurio201369, QIAO2018107,QIAO20181,8118195,QIAO20181a,QIAO2018,QIAO201958}.

	
	\begin{definition} \label{Grouping}\cite{BPMHF12}
		A binary function $G:[0,1]^2 \rightarrow [0,1]$ is said to be a \textbf{grouping function} if it satisfies the following conditions, for all $x, y, z \in [0,1]$:
		\begin{enumerate} [labelindent=\parindent, leftmargin=*,label=\normalfont{(G\arabic*)}]
			\item $G(x,y) = G(y,x)$;\label{G1}
			\item $G(x,y) = 0$ if and only if $x = y = 0$;\label{G2}
			\item $G(x,y) = 1$ if and only if $x = 1$ or $ y = 1$;\label{G3}
			\item if $ x\leq y$ then $G(x,z) \leq G(y, z)$;\label{G4}
			\item $G$ is continuous;\label{G5}
			
		\end{enumerate}
		
	\end{definition}
	
	\begin{remark} \label{Elem.Neutro_group0}
		Note that whenever a grouping function has a neutral element, then, by \ref{G2}, this element is necessarily $0$.
	\end{remark}
	
	For all properties and related concepts on grouping functions, see also \cite{BDBB13,DimuroIPMU,DBS14, Dimuro2017-ql,Jurio201369,QIAO2018107,QIAO20181,QIAO20181a,QIAO201958}.

	\begin{theorem}\cite[Theorem 2]{BPMHF12} \label{OverGrou}
		Let $O$ be an overlap function, and let $N$ be a strict fuzzy negation. Then,
		\begin{eqnarray} \label{Group_of_Over}
		G(x,y)  = N(O(N(x), N(y)))
		\end{eqnarray}
		is a grouping function. Reciprocally, if $G$ is a grouping function, then
		\begin{eqnarray}\label{Over_of_Group}
		O(x,y)  = N(G(N(x), N(y)))
		\end{eqnarray}
		is an overlap function.
	\end{theorem}
	
	In the following proposition we show that if an overlap function $O$ admits a neutral element, then the grouping function $G$ generated by $O$ admits no neutral element.
	\begin{proposition} \label{NeutralONoNeutralG}
		Let $N$ be a strict and non-strong fuzzy negation and $O$ be an overlap function. If $O$ has a neutral element, then the grouping function $G$ given by Eq. (\ref{Group_of_Over}) has no neutral element.
	\end{proposition}
	\begin{proof}
		Since $O$ has a neutral element, then $O(x,1) = x$, for all $x \in [0,1]$. However, as $N$ is a non-strong fuzzy negation, there is $\tilde{x} \in [0,1]$ such that $N(N(\tilde{x})) \neq \tilde{x}$, so:
		\begin{eqnarray*}
			G(\tilde{x},0) = N(O(N(\tilde{x}),1)) = N(N(\tilde{x})) \neq \tilde{x}.
		\end{eqnarray*}
		Therefore, $G$ has no neutral element.
	\end{proof}
	
	\begin{remark}  \label{grouping_of_G_O}
		There are many ways to define a grouping function from other operators. For example, given a grouping function $G$  and an overlap function $O$, we have that $G'(x,y)$  $= G(O(1,x),O(1,y))$ is a grouping function, directly from the properties of $G$ and $O$.
	\end{remark}
	Next, we present the concept of general overlap function.

	\begin{definition}\cite{DeMiguel2019} \label{def-gen-ov}
		A function $ \mathcal{GO}\colon  [0,1]^n$ $ \rightarrow [0,1]$ is said to be a general overlap function if it satisfies the following conditions, for all $\vv{x}= (x_1, \ldots, x_n) \in [0,1]^n$:
		\begin{description}
			\item[($\mathcal{GO}1$)] $\mathcal{GO}(x_1, \cdots, x_n) = \mathcal{GO}(x_{j_1}, \cdots, x_{j_n})$, where \\
			 $(x_{j_1}, \cdots, x_{j_n})$ is a permutation of $(x_1, \cdots, x_n)$;
			
			\item[($\mathcal{GO}2$)] If $\prod_{i=1}^n x_i = 0$ then $\mathcal{GO}(\vv{x})=0$;
			\item[($\mathcal{GO}3$)] If  $\prod_{i=1}^n x_i = 1$   then   $\mathcal{GO}(\vv{x})=1$;
			
			\item[($\mathcal{GO}4$)] $\mathcal{GO}$ is increasing;
			
			\item[($\mathcal{GO}5$)] $\mathcal{GO}$ is continuous.
		\end{description}
		
	\end{definition}

	Some examples of overlap functions and general overlap functions are given in Table \ref{tab-eg}, found in \cite{DeMiguel2019,Dimuro201439}. Observe that any overlap function is a bivariate general overlap function, but the converse does not hold.
	
	
	\subsection{Some new result on general overlap and grouping}
	
	\begin{proposition}
		Let $O$ be an overlap and $a\in (0,1)$. Then $O_a:[0,1]^2\rightarrow [0,1]$ defined, for all $x,y \in [0,1]$,  by
		\begin{eqnarray*} \label{eq-def-GOF-nao-over}
			O_a(x,y)  = \frac{\max(0,O(x,y)-O(\max(x,y),a))}{1-O(\max(x,y),a)}
		\end{eqnarray*}
		is a bivariate general overlap function which is not an overlap function.
	\end{proposition}
	\begin{proof} By \ref{O3}, $O(\max(x,y),a)\neq 1$ and therefore $O_a$ is well defined.
		Clearly, $O_a$ is  commutative, increasing, satisfies ($\mathcal{GO}$2) and ($\mathcal{GO}$3) but does not satisfy  \ref{O2}. In addition, let  $x_i\in [0,1]$ be a sequence such that $\lim_{i\rightarrow \infty} x_i =a$. So, for each $y\in [0,1]$, we have two situations: (i) if $y\leq a$, then  $\lim_{i\rightarrow \infty} O_a(x_i,y)=0=O_a(a,y)$ and (ii) if $y>a$ then, since $O$ is continuous and commutative,  $\lim_{i\rightarrow \infty} O_a(x_i,y)= \lim_{i\rightarrow \infty} \frac{\max(0,O(x_i,y)-O(y,a))}{1-O(y,a)}= 0=O_a(a,y).$ Therefore, $O_a$ is continuous.
	\end{proof}
	
	\begin{table*}[t]
		\centering
			\begin{tabular}{ll}
				\hline
				Overlap functions &   General overlap functions  \\
				\hline
				$O_{mM}(x,y)=\min\{x,y\}\hspace{-0.05cm}\max\{x^2\hspace{-0.1cm},y^2\hspace{-0.05cm}\}$ &  $\mathcal{GO}_{max}(x,y)=\max\{0,x^2+y^2-1\}$\\[0.6cm]
				
				$O_{DB}(x,y)=\left \{ \begin{array}{ll}
				\hspace{-0.1cm} \frac{2xy}{x+y}, & \hspace{-0.1cm} \mbox{ if } \; x+y \neq 0;\\
				0, & \hspace{-0.1cm} \mbox{if }\; x+y=0.
				\end{array}\right.$ &  $\mathcal{GO}_{T_L}(x,y)\hspace{-0.1cm}=\hspace{-0.1cm} (\min\{x,y\})^p \cdot\hspace{-0.05cm} \max\{0,x\hspace{-0.1cm}+\hspace{-0.1cm}y\hspace{-0.1cm}-\hspace{-0.1cm}1\},$ \mbox{for} $p>0$ \\[0.6cm]

				$O_P(x,y)=x^py^p, \mbox{ with } \; p>0.$  & $\mathcal{GO}_{PN}(\vv{x})= \prod\limits_{i=1}^n x_i \cdot \left( \left\{ \begin{array}{ll}
				0, \\ \qquad  \mbox{if } \sum\limits_{i=1}^n x_i \leq 1,\\
				\bigwedge(\vv{x})=\bigwedge(x_1,\ldots,x_n),  \\ \qquad \mbox{otherwise.}
				\end{array}\right.\hspace{-0.2cm} \right)$ \\[0.6cm]

				$O_V (x,y)= \left \{ \begin{array}{ll}
				\frac{1+(2x-1)^2(2y-1)^2}{2},\\   \;\; \mbox{ if } \;\; x,y \in [0.5,1],\\
				\min\{x,y\},  \\ \;\; \mbox{otherwise.}
				\end{array}\right.$ &  $\mathcal{GO}_{GN}(\vv{x})\hspace{-0.1cm}=\hspace{-0.1cm} \sqrt[n]{\prod\limits_{i=1}^n x_i} \cdot \left( \left\{ \begin{array}{ll}
				\hspace{-0.1cm}0, \\ \;  \mbox{ if } \sum\limits_{i=1}^n x_i \leq 1,\\
				\hspace{-0.1cm}\bigwedge(\vv{x})=\bigwedge(x_1,\ldots,x_n), \\  \; \mbox{ otherwise.}
				\end{array}\right.\hspace{-0.2cm} \right)$\\[0.6cm]

				$O_{min}(x,y)=\min\{x,y\}$ & \\[0.6cm]

				\hline
			\end{tabular}
			\caption{Examples of overlap functions $O$  and  general overlap functions $\mathcal{GO}$ }
				\label{tab-eg} 
	\end{table*}
	
\begin{proposition}\label{prop-dual}
	Consider a  strict negation $N$ and a bivariate general overlap function $\mathcal{GO}$. If $\mathcal{GO}$ satisfies the following conditions, for all $x,y \in [0,1]$:
	\begin{description}
		\item [($\mathcal{GO}2a$)]  If $\mathcal{GO}(x,y)=0$ then $xy = 0$;
		\item [($\mathcal{GO}3a$)]  If $\mathcal{GO}(x,y)=1$ then $xy = 1$,
	\end{description}  then
	\begin{eqnarray} \label{Group_of_GOver}
	G(x,y)  = N(\mathcal{GO}(N(x), N(y)))
	\end{eqnarray}
	is a grouping function. Reciprocally, if $G$ is a grouping function, then
	\begin{eqnarray}\label{Over_of_Group}
	\mathcal{GO}(x,y)  = N(G(N(x), N(y)))
	\end{eqnarray}
	is a general overlap function satisfying ($\mathcal{GO}$2a) and ($\mathcal{GO}$3a).
\end{proposition}
\begin{proof}  Since such general overlap function is also an overlap function, then the result follows from   Theorem \ref{OverGrou}.
\end{proof}

An element $a\in [0,1]$ is a neutral element of $\mathcal{GO}$ if for each $x\in [0,1]$, $\mathcal{GO}(x,\underbrace{a,\ldots,a}_{(n-1)-times})=x$.


\begin{proposition} \label{prop-NE-GO}
	Let $\mathcal{GO}$ be a bivariate general overlap function. 1 is a neutral element of  $\mathcal{GO}$ if and only if  $\mathcal{GO}$ satisfies ($\mathcal{GO}$3a) and has a neutral element.
\end{proposition}
\begin{proof}
	If  $\mathcal{GO}(x,y)=1$ then, by ($\mathcal{GO}$4) and since 1 is a neutral element of  $\mathcal{GO}$, one has that $x=\mathcal{GO}(x,1)=1$  and $y=\mathcal{GO}(1,y)=1$, i.e. $xy=1$.
	Conversely, if a bivariate general overlap function $\mathcal{GO}$ satisfies ($\mathcal{GO}$3a) and has a neutral element $a$ then $a=1$. In fact, $\mathcal{GO}(a,1)=1$ and therefore, by ($\mathcal{GO}$3a), $a=1$.
\end{proof}

\begin{remark}
	Observe that the result stated by Proposition \ref{prop-NE-GO} does not mean that when a bivariate general overlap function has a neutral  element then it is equal to 1. In fact, for each $e\in (0,1]$, the function
	\[\mathcal{GO}(x,y)=\left \{ \begin{array}{ll}
	\min(x,y), & \mbox{ if }\max(x,y)\leq e \\
	\max(x,y), & \mbox{ if }\min(x,y)\geq e \\
	\frac{xy}{e}, & \mbox{ if }\min(x,y) <e <\max(x,y)
	\end{array}
	\right. \]
	is a general overlap function with $e$ as neutral element.
\end{remark}

\begin{remark}
	As it is well known, there exists a unique idempotent \textbf{t-norm}, namely, the minimum t-norm. On the other hand, there are uncountable idempotent overlap functions. For example, for every $p, q > 0$, the function
	\begin{eqnarray*}
		O(x, y) = \left ( \dfrac{x^{p}y^{q}+x^{q}y^{p}}{2} \right)^{\frac{1}{p+q}}
	\end{eqnarray*}
	is an idempotent general overlap function, see \cite{BDBB13}.
\end{remark}

\begin{proposition} \label{Ominimum}
	If $1$ is the neutral element of a general overlap function $\mathcal{GO}$ and $\mathcal{GO}$ is idempotent, then $\mathcal{GO}$ is the minimum.
\end{proposition}
\begin{proof}
	Given $x_1,\ldots,x_n\in [0,1]$. Then, since $\mathcal{GO}$ is idempotent and increasing in each variable, 
	\begin{eqnarray*}
		\min(x_1 \hspace{-1.0ex}&,&\hspace{-1.0ex} \ldots,x_n)=\\
		&=&\mathcal{GO}(\min(x_1,\ldots,x_n),\ldots,\min(x_1,\ldots,x_n))\\
		&\leq& \mathcal{GO}(x_1,\ldots,x_n).
	\end{eqnarray*}
	Conversely, for each $i=1,\ldots,n$, since 1 is the neutral element of $\mathcal{GO}$, we have that 
	\begin{eqnarray*}
		\mathcal{GO}(x_1,\ldots,x_n)\leq \mathcal{GO}(x_i,1,\ldots,1)=x_i	
	\end{eqnarray*}
	and therefore, $\mathcal{GO}(x_1,\ldots,x_n)\leq \min(x_1,\ldots,x_n)$.
\end{proof}


\subsection{Fuzzy implications derived from overlap and grouping functions}

One can find the definition of fuzzy implication functions in \cite{fodor1994,BJ08,Mas07a}, given as follows:

\begin{definition} \label{ImpFuzzy} 
	A function $I: [0,1]^{2} \rightarrow [0,1]$ is a \textbf{fuzzy implication} if, for all $x, y, z \in [0,1] $, the following properties are satisfied:
	\begin{enumerate} [labelindent=\parindent, leftmargin=*,label=\normalfont{(I\arabic*)}]
		\item If $ x\leq z$ then $I(x,y) \geq I(z,y)$;\label{I1}
		\item If $ y\leq z$ then $I(x,y) \leq I(x,z)$;\label{I2}
		\item  $I(0,y) = 1$;\label{I3}
		\item $I(x,1) = 1$;\label{I4}
		\item $I(1,0) = 0$.\label{I5}
		
	\end{enumerate}
\end{definition}

We denote by $\mathcal{FI}$ the set of all fuzzy implications.

\begin{definition}\cite[Def. 11]{PBSS18b}
	A fuzzy implication function is said to be crisp if  $I(x,y)\in \{0,1\}$,  for each $x,y\in [0,1]$.
\end{definition}

\begin{proposition}  \cite[Prop. 2]{PBSS18b}  Let $I : [0, 1]^2 \rightarrow [0, 1]$ be a fuzzy implication. Then $I$ is crisp if and only if one of the following conditions are satisfied, for all $x, y \in [0, 1]$:
	\begin{description}
		\item [(C1)] If there exists $\alpha \in (0, 1]$ and $\beta \in  [0, 1)$ such that $I = I_{\underline{\alpha}}^{\overline{\beta}}$,
		where
		\[I_{\underline{\alpha}}^{\overline{\beta}}(x, y) = 
		\begin{cases}
		0, & \mbox{ if} \; x \geq \alpha \; \mbox{and} \; y \leq \beta \\
		1, & \mbox{ otherwise}.
		\end{cases}  \]
		
		\item [(C2)] If there exists $\alpha \in [0, 1)$ and $\beta \in  (0, 1]$ such that $I = I_{\alpha}^{\beta}$,
		where
		\[I_{\alpha}^{\beta}(x, y) = 
		\begin{cases}
		0, & \mbox{if} \; x > \alpha \; \mbox{and} \; y < \beta \\
		1, & \mbox{otherwise}.
		\end{cases}  \]

		\item [(C3)] If there exists $\alpha, \beta  \in (0, 1]$  such that $I = I_{\underline{\alpha}}^{\beta}$,
		where
		\[I_{\underline{\alpha}}^{\beta}(x, y) = 
		\begin{cases}
		0, & \mbox{if} \; x \geq \alpha \; \mbox{and} \; y < \beta \\
		1, & \mbox{otherwise}.
		\end{cases} \]

		\item [(C4)] If there exists $\alpha, \beta  \in [0, 1)$  such that $I = I_{\alpha}^{\overline{\beta}}$,
		where
		\[I_{\alpha}^{\overline{\beta}}(x, y) = 
		\begin{cases}
		0, & \mbox{ if} \; x > \alpha \; \mbox{and} \; y \leq \beta \\
		1, & \mbox{ otherwise}.
		\end{cases} \]
	\end{description}
\end{proposition}

\begin{definition} \cite{BJ08}
	Let $I \in \mathcal{FI}$. The function $N_{I}: [0,1] \rightarrow [0,1]$ defined by
	\begin{equation} \label{NegNatural}
	N_{I}(x)= I(x, 0), \ \ x \in [0,1]
	\end{equation}
	is called \textbf{natural negation} of $I$ or negation induced by $I$.
	
\end{definition}

Observe that $N_I$ is in fact a fuzzy negation and, in case $I$ is crisp then $N_I$ is a crisp fuzzy negation. Other properties can be required for fuzzy implications. In the following, we present some that are considered in this paper:
\begin{definition}\label{prop.IF}
	A fuzzy implication function $I$ satisfies, for all $x,y,z \in [0,1]$,  the:
	
	\begin{quote}	
		\begin{itemize}
			\item[(NP)] \textbf{Left neutrality property} if and only if $\ \forall y \in [0,1] \colon \ I(1, y) = y$;
			
			\item[(IP)] \textbf{Identity principle} if and only if 
			$\forall x \in [0,1] \colon$ $I(x, x) = 1$;
			
			\item[(EP)] \textbf{Exchange principle} if and only if $\, \forall  x, y, z \in [0,1] \colon \ I(x, I(y, z)) = I(y, I(x, z))$;
			
			\item[(EP1)] \textbf{Exchange principle for} $\mathbf{1}$ if and only if $\, \forall  x, y, z$  $\in [0,1] \colon I(x, I(y, z)) = 1 \Rightarrow  I(y, I(x, z)) = 1$;
			
			\item[(IB)] \textbf{Iterative Boolean law} if and only if $ \forall  x, y \in [0,1]\colon$ $I(x, I(x,y)) = I(x, y)$;
			
			\item[(LOP)] \textbf{Left-ordering property}, if, for all $x, y \in [0,1] \colon$ $ I(x, y) = 1 $ whenever $x \leq y;$

			\item[(ROP)] \textbf{Right-ordering property}, if, for all $x, y \in [0,1] \colon$ $ I(x, y) \neq 1 $ whenever $x > y$.  

			\item[(CP)] \textbf{Law of contraposition} (or in other words, the contrapositive symmetry) with respect to fuzzy negation $N$, if and only if $\ \forall  x, y \in [0,1] \colon$
			
			$I(x, y) = I(N(y), N(x))$;
			
			\item[(L-CP)] \textbf{Law of left contraposition} with respect to fuzzy negation $N$ if and only if $\, \forall  x, y \in [0,1] \colon$
				\begin{eqnarray*} \label{LCP}
					I(N(x), y) = I(N(y), x);
				\end{eqnarray*} 
			
			\item[(R-CP)] \textbf{Law of right contraposition} with respect to fuzzy negation $N$ if and only if $\ \forall  x, y \in [0,1] \colon$
				\begin{eqnarray*}
					I(x, N(y)) = I(y, N(x)).
				\end{eqnarray*} 
		\end{itemize}	
	\end{quote}
	
\end{definition}

If $I$ satisfies the (left, right) contrapositive symmetry with respect to $N$, then we also denote this by \emph{CP(N)}, respectively, by $L{-}CP(N)$ and $R{-}CP(N)$.

It is well known that from binary operations on the unit interval [0,1], for instance, from t-norms, t-conorms ($N$-dual of t-norms) and fuzzy negations, it is possible to obtain families of fuzzy implications, \cite{BJ08}.  Nevertheless, we can also use overlap and grouping functions to obtain other families of implication functions, such as $(G,N),QL$ and $R_{O}$-implication functions, defined as follows.

\begin{definition}\label{def-impl-varias}
	Let $O$ be an overlap function, $G$ be a grouping function and $N$ be a fuzzy negation. Then, the functions $I_{G,N},I_{O,G,N_{\top}},I_O,I_G^D:[0, 1]^2 \rightarrow [0, 1]$ are called:
	\begin{enumerate}
		\item $\mathbf{(G,N)}$\textbf{-implication}, given by \cite{DBS14}, if
		\begin{eqnarray}\label{eq-GN-impl}
		I_{G,N}(x,y) = G(N(x),y);
		\end{eqnarray}
		\item $\mathbf{QL}$\textbf{-implication}, given by \cite{Dimuro2017-ql}, if
		\begin{eqnarray*}
			I_{O,G,N_{\top}} (x,y)        =  \left\{ \begin{array}{ll}
				G(0,O(1,y)) & \mbox{if} \; x =1\\
				1  & \mbox{if} \; x <1;
			\end{array} \right.
		\end{eqnarray*}
		\item A residual $\mathbf{R_O}$\textbf{-implication}, given by  \cite{Dimuro2015}, if
		\begin{eqnarray*}
			I_{O}(x,y) = \max\{ z \in [0,1] \mid O(x,z) \leq y \};
		\end{eqnarray*}
		\item $D$\textbf{-implication} derived from $G$, given by \cite{d-impl}, if
		\begin{eqnarray*}\label{eq-D-top}
			I_{G}^D (x,y)= \left \{ \begin{array}{ll}
				G(0,y) & \mbox{if} \; x=1, \\
				1 & \mbox{otherwise.}
			\end{array} \right.\end{eqnarray*}
	\end{enumerate}
\end{definition}

\section{$(\mathcal{GO},N)$-Implications}

In \cite{PBSS17,Vania-Nafips18,PBSS18}, a class of fuzzy implication was investigated, named  $(T,N)$-implications which were introduced in \cite{Bedregal07} and were constructed from the composition of a fuzzy negation and a t-norm. In those works various properties of $(T,N)$-implications were also discussed.

In this sense, we now study an analogous class of implication by replacing the t-norm by a bivariate general  overlap function. Thus, we provide a new class of implication function called $(\mathcal{GO},N)$-implications, defined as follows.

\begin{definition}
	A function $I:[0, 1]^2 \rightarrow [0, 1]$ is called a $(\mathcal{GO},N)$\textbf{-implication} if there exist a bivariate general  overlap function $\mathcal{GO}$ and a fuzzy negation $N$ such that
	\begin{equation}\label{NossaIMP_FOverlap}
	I(x,y) = N(\mathcal{GO}(x, N(y))),
	\end{equation}
	for all $x,y \in [0,1]$. If $N$ is strict, then $I$ is called strict $(\mathcal{GO},N)$-implication. Analogously, if $N$ is strong, $I$ is called strong $(\mathcal{GO},N)$-implication.
\end{definition}

From now on, if $I$ is a $(\mathcal{GO},N)$-implication function generated from $\mathcal{GO}$ and $N$, then we will denote that function by $I_{\mathcal{GO}}^{N}$.

\begin{example}
	We can construct some examples of $I_{\mathcal{GO}}^{N}$.
	
	\begin{enumerate}
		\item [(i)] Consider the general overlap function:
		
		$\mathcal{GO}_{max}(x,y)=\max\{0,x^2+y^2-1\}$ and the standard fuzzy negation $N_Z(x)=1-x$, then we have that:
		\begin{eqnarray*}
			I_{\mathcal{GO}_{max}}^{N_Z}(x,y)=\min\{1, 1-x^2-y^2+2y \}
		\end{eqnarray*}
		\item [(ii)] Consider the general overlap function:
		
		$\mathcal{GO}^{T_L}(x,y)= (\min\{x,y\})^p \ast \max\{0,x+y-1\}$, for $p=2$ and $N_Z(x)=1-x$, then we have that:
		\begin{eqnarray*}
			I_{\mathcal{GO}^{T_L}}^{N_Z}\hspace{-1.5ex}&(&\hspace{-1.0ex}x,y)= 1\hspace{-0.5ex}-\hspace{-0.5ex}(\min\{x^2, y^2\hspace{-0.5ex}-\hspace{-0.5ex}2y\hspace{-0.5ex}+\hspace{-0.5ex}1 \} \ast \max\{0,x\hspace{-0.5ex}-\hspace{-0.5ex}y\} )
		\end{eqnarray*}
		\item [(iii)] Consider the general overlap function $\mathcal{GO}_{max}$ and the crisp fuzzy negation $N_{\alpha}$, then we have that:	
		\begin{eqnarray*}
			I_{\mathcal{GO}_{max}}^{N_{\alpha}}(x,y)= \left \{ \begin{array}{ll}
				0, & \mbox{if } \; y \leq \alpha \mbox{ and } x^2 > \alpha. \\
				1, & \mbox{if } y > \alpha \mbox{, or } y \leq \alpha \mbox{ and } x^2 \leq \alpha.
			\end{array} \right.\end{eqnarray*}
		
		\item [(iv)] Consider the general overlap function $\mathcal{GO}^{T_L}$, for $p=2$ and the crisp fuzzy negation $N^{\alpha}$, then we have that:
		\begin{eqnarray*}
			I_{\mathcal{GO}^{T_L}}^{N^{\alpha}}(x,y)= \left \{ \begin{array}{ll}
				0, & \mbox{if } \; y < \alpha \mbox{ and } x^3 \geq \alpha.\\
				1, & \mbox{if } y \geq \alpha \mbox{, or } y < \alpha \mbox{ and } x^3 < \alpha.
			\end{array} \right.\end{eqnarray*}
	\end{enumerate}
	
\end{example}

\begin{proposition} \label{nossaIFOverlap}
	If $I$ is a $(\mathcal{GO},N)$-implication then $I \in \mathcal{FI}$.\end{proposition}
\begin{proof}
	Indeed, let $I$ be a $(\mathcal{GO},N)$-implication function generated by a general overlap function $\mathcal{GO}$ and a fuzzy negation $N$, then
	
	\noindent $\mathbf{\ref{I1}}$ Given $x, y \in [0,1]$ such that $ x \leq y$, by ($\mathcal{GO}$4), for all $z \in [0,1]$, it holds that $\mathcal{GO}(x,N(z)) \leq \mathcal{GO}(y,N(z))$. So, $N(\mathcal{GO}(y,N(z))) \leq N(\mathcal{GO}(x,N(z)))$, that is, $I_{\mathcal{GO}}^{N}(y,z) \leq I_{\mathcal{GO}}^{N}(x,z)$.
	
	\noindent $\mathbf{\ref{I2}}$ Analogous to $\mathbf{\ref{I1}}$.
	
	\noindent $\mathbf{\ref{I3}}$ For all $y \in [0,1]$, $I_{\mathcal{GO}}^{N}(0,y) = N(\mathcal{GO}(0,N(y))) \stackrel{(\mathcal{GO}2)}{=} N(0) = 1$.
	
	\noindent $\mathbf{\ref{I4}}$ For all $x \in [0,1]$, $I_{\mathcal{GO}}^{N}(x,1) = N(\mathcal{GO}(x,N(1))) = N(\mathcal{GO}(x,0)) \stackrel{(\mathcal{GO}2)}{=} N(0) = 1$.
	
	\noindent $\mathbf{\ref{I5}}$ $I_{\mathcal{GO}}^{N}(1,0) = N(\mathcal{GO}(1,N(0))) = N(\mathcal{GO}(1,1)) \stackrel{(\mathcal{GO}3)}{=} N(1) = 0$.
	
	\noindent Therefore, $I_{\mathcal{GO}}^{N}$ is a fuzzy implication function. 				
\end{proof}


\begin{proposition}
	Let $N$ be a strict fuzzy negation and $\mathcal{GO}$ be a general overlap function. If $\mathcal{GO}$ has no neutral element, then $I_{\mathcal{GO}}^{N} \neq I_{T}^{N}$ for all t-norm $T$.
\end{proposition}

\begin{proof}
	By hypothesis, $\mathcal{GO}$ has no neutral element, so there is $\tilde{y} \in (0,1)$ such that $\mathcal{GO}(1, \tilde{y}) \neq \tilde{y}$. Since $N$ is strict, given $\tilde{y} \in (0,1)$, there is $\tilde{x} \in (0,1)$ such that $N(\tilde{x}) = \tilde{y}$. So,
	\begin{eqnarray*}
		\mathcal{GO}(1, N(\tilde{x})) \hspace{-0.5ex} \neq \hspace{-0.5ex} N(\tilde{x}) \hspace{-0.7ex}&\stackrel{N \ \text{strict}}{\Rightarrow}&\hspace{-0.7ex} N(\mathcal{GO}(1, N(\tilde{x}))) \neq N(N(x))\\
		&\Rightarrow& I_{\mathcal{GO}}^{N}(1, N(\tilde{x})) \neq N(N(\tilde{x})).
	\end{eqnarray*}
	On the other hand, for all t-norm $T$,
	\begin{eqnarray*}
		I_{T}^{N}(1, N(\tilde{x})) &=& N(T(1, N(\tilde{x}))) = N(N(\tilde{x}))\\
		& \neq& I_{\mathcal{GO}}^{N}(1, N(\tilde{x})).
	\end{eqnarray*}
	Therefore, $I_{\mathcal{GO}}^{N} \neq I_{T}^{N}$.
\end{proof}

\begin{example}
	Consider the general overlap function and the strict fuzzy negation given by $\mathcal{GO}_{max}(x,y)=\max\{0,x^2+y^2-1\}$ and $N(x)=1-x^2$, respectively, then we have that:
	\begin{eqnarray*}
		I_{\mathcal{GO}_{max}}^{N}(x,y)&=& N(\mathcal{GO}_{max}(x, N(y)))\\
		&=& 1 - \big(\max\{0,x^2+(1-y^2)^2-1\}\big)^2.
	\end{eqnarray*}
	Observe that $I_{\mathcal{GO}_{max}}^{N}(1,y) = 1 - \big(\max\{0,1+(1-y^2)^2-1\}\big)^2 = 1 - \big(1-y^2\big)^4 $ and for all t-norm $T$, $I_{T}^{N}(1,y) = N(T(1, N(y))) = N(N(y)) = 1- (1- y^2)^2$. Therefore, for all $y \in (0,1)$, $I_{\mathcal{GO}_{max}}^{N}(1,y) \neq I_{T}^{N}(1,y)$.
\end{example}

Observe that it is possible to recover the bivariate general overlap function from any $(\mathcal{GO},N)$-implication function which was constructed from such general overlap function and a strict fuzzy negation, as shown in the following proposition:

\begin{proposition}\label{prop-GO-from-GON-Imp}
	Let $\mathcal{GO}$ be a bivariate general overlap function and $N$ be a fuzzy negation. If $ N $ is strict, then
	$\mathcal{GO}(x,y) = N^{-1}(I_{\mathcal{GO}}^{N}(x, N^{-1}(y)))$,
	for all $x, y \in [0,1]$.
\end{proposition}
\begin{proof}
	Straightforward.
\end{proof}

\begin{corollary}
	Let $\mathcal{GO}$ be a bivariate general overlap function and $N$ be a fuzzy negation. If $ N $ is strong, then $\mathcal{GO}(x,y) = N(I_{\mathcal{GO}}^{N}(x, N(y)))$, for all $x, y \in [0,1]$.
\end{corollary}

\begin{proposition}\label{prop-neutral-e}
	Let $\mathcal{GO}$ and $N$ be a bivariate general overlap function and a fuzzy negation, respectively. Then,
	\begin{enumerate}
		\item [(i)] If $1$ is the neutral element of $\mathcal{GO}$, then $N_{I_{\mathcal{GO}}^{N}} = N$;
		\item [(ii)] If $N$ is strict and $N_{I_{\mathcal{GO}}^{N}} = N$, then $1$ is the neutral element of $\mathcal{GO}$.
	\end{enumerate}
	
\end{proposition}
\begin{proof}
	Indeed,
	\begin{enumerate}
		\item [(i)] $\forall x \hspace{-0.1cm}\in \hspace{-0.1cm} [0,1]$, $N_{I_{\mathcal{GO}}^{N}}(x)\hspace{-0.1cm} =\hspace{-0.1cm} I_{\mathcal{GO}}^{N}(x,0)\hspace{-0.1cm} =\hspace{-0.1cm} N(\mathcal{GO}(x, N(0)))\hspace{-0.1cm} = \hspace{-0.1cm} N(\mathcal{GO}(x,1))\hspace{-0.1cm} = \hspace{-0.1cm}N(x)$.
		
		\item [(ii)] Since $N$ is strict and $N_{I_{\mathcal{GO}}^{N}} = N$, for all $x \in [0,1]$, one has that: 
		\begin{eqnarray*}
			\mathcal{GO}(x,1)\hspace{-0.1cm} &=& \hspace{-0.1cm} N^{-1}(N(\mathcal{GO}(x,N(0)))) = N^{-1}(I_{\mathcal{GO}}^{N}(x,0))\\
			& =& \hspace{-0.1cm} N^{-1}(N(x))\hspace{-0.1cm} = \hspace{-0.1cm}x.
		\end{eqnarray*}
	\end{enumerate}	
\end{proof}

Note that the converse of Prop. \ref{prop-neutral-e}(i) is not always true, i.e. there are non-strict negations $N$ that satisfy $N_{I_{\mathcal{GO}}^{N}} = N$, but $\mathcal{GO}$ has no neutral element. See the following example:

\begin{example}
	Take the fuzzy negation $N_{\top}$ given by \[N_{\top}(x) = \begin{cases}
	0, & \text{ if }  x = 1 \\
	1, & \text{ if }  x \neq 1
	\end{cases}, \] and consider a bivariate  general overlap function $\mathcal{GO}$ that satisfies ($\mathcal{GO}$3a). Then, for all $x \in [0,1]$, one has that:
	\begin{eqnarray*}
		N_{I_{\mathcal{GO}}^{N_\top}}(x) &=& I_{\mathcal{GO}}^{N_\top}(x,0) = 
		N_\top(\mathcal{GO}(x,1))\\
		&=& \begin{cases}
			0, & \text{ if }  \mathcal{GO}(x,1) = 1 \\
			1, & \text{ if }  \mathcal{GO}(x,1) \neq 1
		\end{cases}\\
		&\stackrel{(\mathcal{GO}\text{3a})}{=}&  \begin{cases}
			0, & \text{ if }  x = 1 \\
			1, & \text{ if }  x \neq 1
		\end{cases}
		= N_\top(x).
	\end{eqnarray*}
	However,  $\mathcal{GO}$ does not   necessarily have  a neutral element.
\end{example}

\begin{proposition}
	Let $\mathcal{GO}$ be a bivariate general overlap function and $N$ be a fuzzy negation such that $x \leq N(N(x))$, for all $x \in [0,1]$. Then:
	\begin{enumerate}
		\item [(i)] If $\mathcal{GO}(1,y) \leq y$, then $y \leq I_{\mathcal{GO}}^{N}(x, y)$;
		\item [(ii)] If $N$ is strict and $y \leq I_{\mathcal{GO}}^{N}(x, y)$, then $\mathcal{GO}(1,y) \leq y$.
	\end{enumerate}
	
\end{proposition}
\begin{proof}
	Indeed,
	\begin{enumerate}
		\item [(i)] By hypothesis, take $\mathcal{GO}(1,N(y)) \leq N(y)$. Then, applying $N$ on both sides, $N(N(y)) \leq N(\mathcal{GO}(1, N(y)))$. On the other hand,
		\begin{eqnarray*}
			x \leq 1 &\stackrel{(\mathcal{GO}4)}{\Rightarrow}& \mathcal{GO}(x, N(y)) \leq \mathcal{GO}(1, N(y))\\
			&\Rightarrow& N(\mathcal{GO}(1, N(y))) \leq N(\mathcal{GO}(x, N(y)))
		\end{eqnarray*}
		for all $x, y \in [0,1]$. So, it follows that $y \leq N(N(y)) \leq N(\mathcal{GO}(1, N(y))) \leq N(\mathcal{GO}(x, N(y)))$, and, therefore,  $y \leq I_{\mathcal{GO}}^{N}(x, y)$.
		
		\item [(ii)] Since $y \leq I_{\mathcal{GO}}^{N}(x, y)$, for all $x, y \in [0,1]$, so, in particular, $y \leq I_{\mathcal{GO}}^{N}(1, y)$, for all $y \in [0,1]$. Moreover, $y \leq N(\mathcal{GO}(1, N(y))) \stackrel{\ref{N1}}{\Rightarrow} N(N(\mathcal{GO}(1, N(y)))) \leq N(y)$, hence, by hypothesis,  
		\begin{eqnarray*}
			\mathcal{GO}(1, N(y)) \leq N(N(\mathcal{GO}(1, N(y)))) \leq N(y),
		\end{eqnarray*} 
		for all $y \in [0,1]$. So, $\mathcal{GO}(1,y) = \mathcal{GO}(1, N(N^{-1}(y))) \leq N(N^{-1}(y)) = y$, since $N$ is strict. Therefore, for all $y \in [0,1]$, $\mathcal{GO}(1,y) \leq y$.
		
	\end{enumerate}
	
\end{proof}

\begin{proposition} \label{Prop.(O,N)}
	Let $I_{\mathcal{GO}}^{N}$ be a $(\mathcal{GO},N)$-implication. Then:
	\begin{enumerate}
		\item [(i)] $I_{\mathcal{GO}}^{N}$ satisfies L-CP(N);
		\item [(ii)] If $N$ is a strict negation, then $I_{\mathcal{GO}}^{N}$ satisfies R-CP(N$^{-1}$).
		\item [(iii)] If $I_{\mathcal{GO}}^{N}$ satisfies R-CP(N) with a strict negation $N$ and $1$ is the neutral element of $\mathcal{GO}$, then $N$ is a strong negation;
		\item [(iv)] If $N$ is a strong negation, then $I_{\mathcal{GO}}^{N}$ satisfies CP(N).
		\item [(v)] If $I_{\mathcal{GO}}^{N}$ satisfies CP(N) with a strict negation $N$  and $1$ is the neutral element of $\mathcal{GO}$, then $N$ is a strong negation.
	\end{enumerate}
\end{proposition}
\begin{proof}
	\begin{enumerate}
		\item [(i)] For all $x,y \in [0,1]$, it holds that:
		\begin{eqnarray*}
			I_{\mathcal{GO}}^{N}(N(x), y) &=& N(\mathcal{GO}(N(x), N(y))) \\
			&\stackrel{(\mathcal{GO}1)}=& N(\mathcal{GO}(N(y), N(x)))\\
			&=& I_{\mathcal{GO}}^{N}(N(y), x).
		\end{eqnarray*}
		\item [(ii)] For all $x,y \in [0,1]$, one has that:
		\begin{eqnarray*}
			I_{\mathcal{GO}}^{N}(x, N^{-1}(y)) &=& N(\mathcal{GO}(x, N(N^{-1}(y))))\\
			& =& N(\mathcal{GO}(x, y)) \stackrel{(\mathcal{GO}1)}{=} N(\mathcal{GO}(y, x))\\
			& =& N(\mathcal{GO}(y, N(N^{-1}(x))))\\
			&=& I_{\mathcal{GO}}^{N}(y, N^{-1}(x)).
		\end{eqnarray*}
		\item [(iii)] Since $I_{\mathcal{GO}}^{N}$ satisfies $R-CP(N)$, then $I_{\mathcal{GO}}^{N}(1,N(y)) = I_{\mathcal{GO}}^{N}(y, N(1))$. Hence, since $N$ is a strict negation, 
		\begin{eqnarray*}
			\mathcal{GO}(1, N(N(y))) = \mathcal{GO}(y, N(N(1)))	
		\end{eqnarray*}
		for all $y \in [0,1]$, i.e., $\mathcal{GO}(1, N(N(y))) = \mathcal{GO}(y, 1)$ for all $y \in [0,1]$. So, since $1$ is the neutral element of $\mathcal{GO}$, $N(N(y)) = y$, for all $y \in [0,1]$.
		\item [(iv)] For all $x,y \in [0,1]$, since $N$ is strong, it follows that:
		\begin{eqnarray*}
			I_{\mathcal{GO}}^{N}(N(y), N(x)) \hspace{-1.0ex}&=& \hspace{-1.0ex} N(\mathcal{GO}(N(y), N(N(x))))\\
			&=& \hspace{-1.0ex}N(\mathcal{GO}(N(y), x)) \\
			&\stackrel{(\mathcal{GO}1)}{=}&\hspace{-1.0ex} N(\mathcal{GO}(x, N(y))) = I_{\mathcal{GO}}^{N}(x, y).
		\end{eqnarray*}
		 
		\item [(v)] Since $I_{\mathcal{GO}}^{N}$ satisfies CP(N) and $N$ is a strict negation, then $\mathcal{GO}(x, N(0)) = \mathcal{GO}(N(0), N(N(x)))$ for all $x \in [0,1]$, i.e., $\mathcal{GO}(x, 1) = \mathcal{GO}(1, N(N(x)))$ for all $x \in [0,1]$. So, since $1$ is the neutral element of $\mathcal{GO}$, $N(N(x))$ $= x$, for all $x \in [0,1]$.	
	\end{enumerate}
\end{proof}

\begin{proposition}
	Let $I_{\mathcal{GO}}^{N}$ be a $(\mathcal{GO},N)$-implication. If $N$ is a strong negation. Then:
	\begin{enumerate}
		\item [(i)] $I_{\mathcal{GO}}^{N}$ satisfies (NP) if and only if $1$ is the neutral element of $\mathcal{GO}$.
		\item [(ii)] $I_{\mathcal{GO}}^{N}$ satisfies (EP) if and only if $\mathcal{GO}$ is associative.
	\end{enumerate}
\end{proposition}

\begin{proof}
	Indeed,
	\begin{enumerate}
		\item [(i)]
		Consider $I_{\mathcal{GO}}^{N}(1,y) = y$, for all $y \in [0,1]$. Then, since $N$ is strong,
		\begin{center}
			$\mathcal{GO}(1, N(y)) =  N(y)$ \hfil(*)
		\end{center}
		for all $y \in [0,1]$. So, one has that
		\[\mathcal{GO}(1, x) \stackrel{(N5)}{=} \mathcal{GO}(1, N(N(x))) \stackrel{(*)}{=} N(N(x)) \stackrel{(N5)}{=} x,\] for all $x \in [0,1]$. Conversely, since $1$ is neutral element of $\mathcal{GO}$, then for all $y \in [0,1]$, 
		\begin{eqnarray*}
			I_{\mathcal{GO}}^{N}(1,y) = N(\mathcal{GO}(1, N(y))) = N(N(y)) \stackrel{(N5)}{=} y.
		\end{eqnarray*}
		
		\item [(ii)] Consider that $I_{\mathcal{GO}}^{N}$ satisfies (EP). Then, for all $x,y,z \in [0,1]$, since $N$ is a strong negation,
		\begin{eqnarray*}
			N\hspace{-1.0ex}&(&\hspace{-1.0ex} \mathcal{GO}( x,\mathcal{GO}(y,z))) = \\
			&=&  N(\mathcal{GO}(x,N(N(\mathcal{GO}(y,N(N(z)))))))\\
			&=&  I_{\mathcal{GO}}^{N}(x,I_{\mathcal{GO}}^{N}(y,N(z))) = I_{\mathcal{GO}}^{N}(y,I_{\mathcal{GO}}^{N}(x,N(z)))\\
			&=& N(\mathcal{GO}(y,\mathcal{GO}(x,z)))
		\end{eqnarray*} 
		and so, $\mathcal{GO}(x,\mathcal{GO}(y,z)) =\mathcal{GO}(y,\mathcal{GO}(x,z)$,
		for all $x,y,z \in [0,1]$.
		Therefore, $\mathcal{GO}$ is associative.
		
		Conversely, for all $x,y,z \in [0,1]$, since $N$ is strong and $\mathcal{GO}$ is associative then
		\begin{eqnarray*}
			I_{\mathcal{GO}}^{N}(x,I_{\mathcal{GO}}^{N}(y,z)) &=&
			N(\mathcal{GO}(x,\mathcal{GO}(y,N(z)))) \\
			&\stackrel{\mathcal{GO} \ \text{Assoc.}}{=} & N(\mathcal{GO}(\mathcal{GO}(x,y),N(z))) \\
			&\stackrel{(\mathcal{GO}1)}{=}& N(\mathcal{GO}(\mathcal{GO}(y,x),N(z))) \\
			& \stackrel{\mathcal{GO} \ \text{Assoc.}}{=} & N(\mathcal{GO}(y,\mathcal{GO}(x,N(z)))) \\
			& = & I_{\mathcal{GO}}^{N}(y,I_{\mathcal{GO}}^{N}(x,z)).
		\end{eqnarray*}
		Therefore, $I_{\mathcal{GO}}^{N}$ satisfies (EP).
	\end{enumerate}
\end{proof}

\begin{proposition} Let $\mathcal{GO}$ be a bivariate general overlap function satisfying  ($\mathcal{GO}$2a), and  $I_{\mathcal{GO}}^{N}$ be a $(\mathcal{GO},N)$-implication. If $N$ is a frontier fuzzy negation, then $I_{\mathcal{GO}}^{N}$ satisfies (EP1).
\end{proposition}

\begin{proof}
	Suppose that $I_{\mathcal{GO}}^{N}(x, I_{\mathcal{GO}}^{N}(y,z)) = 1$, for $x, y, z \in [0,1]$. This means that $N(\mathcal{GO}(x,N(N(\mathcal{GO}(y,N(z)))))) = 1$. In this case, since $N$ is a frontier negation, then:
	
	\noindent $\mathcal{GO}(x,N(N(\mathcal{GO}(y,N(z))))) =0$. 
	
	\noindent By  ($\mathcal{GO}$2a), this means that  $x=0$ or $N(N(\mathcal{GO}(y,N(z))))$ $= 0$. Then, one has the following cases:
	
	\noindent (1) For $x=0$, it follows that:
	\begin{eqnarray*}
		I_{\mathcal{GO}}^{N}(y, I_{\mathcal{GO}}^{N}(0,z)) = 
		I_{\mathcal{GO}}^{N}(y,1) = N(\mathcal{GO}(y,0)) \stackrel{(\mathcal{GO}2)}{=} 1.
	\end{eqnarray*}
		
	\noindent	(2) For $N(N(\mathcal{GO}(y,N(z)))) = 0$, since $N$ is a frontier negation, so 
	$\mathcal{GO}(y,N(z))$ $ = 0$.		
	So, by ($\mathcal{GO}$2a), $y=0$ or $z=1$. If $y=0$, then $I_{\mathcal{GO}}^{N}(0, I_{\mathcal{GO}}^{N}(x,z)) =1$. 
	On the other hand, if $z=1$, then
	$I_{\mathcal{GO}}^{N}(y, I_{\mathcal{GO}}^{N}(x,1)) = I_{\mathcal{GO}}^{N}(y,1)=1$.
	
	Thus, in any case, it holds that $I_{\mathcal{GO}}^{N}(y, I_{\mathcal{GO}}^{N}(x,z)) = 1$.		
\end{proof}

\begin{proposition}\label{prop-idp}
	Let $I_{\mathcal{GO}}^{N}$ be a $(\mathcal{GO},N)$-implication with a strict fuzzy negation $N$. If $I_{\mathcal{GO}}^{N}$ satisfies $(IB)$ and $\mathcal{GO}$ has 1 as neutral element, then $N$ is strong and  $\mathcal{GO}$ is idempotent.
\end{proposition}
\begin{proof}
	Indeed, since $I_{\mathcal{GO}}^{N}$ satisfies $(IB)$, we have for $x=1$, $I_{\mathcal{GO}}^{N}(1,I_{\mathcal{GO}}^{N}(1,y)) = I_{\mathcal{GO}}^{N}(1,y)$, $\forall y \in [0,1]$. So, 
	\begin{eqnarray*}
		N(\mathcal{GO}(1,N(N(\mathcal{GO}(1,N(y)))))) = N(\mathcal{GO}(1,N(y)))
	\end{eqnarray*}
	and therefore,
	$N(N(N(N(y)))) = N(N(y))$, for all $y \in [0,1]$, since 1 is neutral element of $\mathcal{GO}$. However, $N$ being a strict negation, then
	$N(N(y)) = y$, for all $y \in [0,1]$ and, then, $N$ is strong. Moreover, since
	\begin{eqnarray*} 
		\lefteqn{I_{\mathcal{GO}}^{N}(x,I_{\mathcal{GO}}^{N}(x,N(y))) = I_{\mathcal{GO}}^{N}(x,N(y))},
	\end{eqnarray*}
	we have that $N(\mathcal{GO} (x, \mathcal{GO}(x, (y))))= N (\mathcal{GO}(x, y))$, since $N$ is strong. So, $\mathcal{GO}(x,\mathcal{GO}(x,y)) = \mathcal{GO}(x, y)$.
	In particular, for $y=1$, 
	$\mathcal{GO}(x,x) = x$, for all $x \in [0,1]$, since $1$ is the neutral element of $\mathcal{GO}$. Therefore, the general overlap function $\mathcal{GO}$ is idempotent.		
\end{proof}

\begin{corollary}
	Let $I_{\mathcal{GO}}^{N}$ be a $(\mathcal{GO},N)$-implication with a strict fuzzy negation $N$. If $I_{\mathcal{GO}}^{N}$ satisfies (IB) and $1$ is the neutral element of the bivariate general overlap function $\mathcal{GO}$, then $\mathcal{GO}$ is the minimum t-norm.
\end{corollary}
\begin{proof} Straightforward from Propositions \ref{prop-idp} and \ref{Ominimum}.	
\end{proof}


\begin{remark}\label{rem-Crisp-GON-imp}
	Observe that, trivially, $I_{\mathcal{GO}}^{N}$ is crisp if and only if $N$ is crisp. In fact, for each $\alpha\in (0,1)$, if 1 is a neutral element of $\mathcal{GO}$ then 	 $I_{\mathcal{GO}}^{N^\alpha}=I_{\underline{\alpha}}^{\alpha}$ and $I_{\mathcal{GO}}^{N_\alpha}=I_{\alpha}^{\overline{\alpha}}$.
\end{remark}

\begin{proposition}
	Let $I_{\mathcal{GO}}^{N}$ be a crisp $(\mathcal{GO},N)$-implication, and let 1 be a neutral element of $\mathcal{GO}$, then:
	\begin{enumerate}
		\item [(i)] $I_{\mathcal{GO}}^{N}$ satisfies (EP) but it does not satisfy (NP).
		\item [(ii)] $I_{\mathcal{GO}}^{N}$ satisfies (LOP) but it does not satisfy (ROP);
		\item [(iii)] $I_{\mathcal{GO}}^{N}$ satisfies (IP);
		\item [(iv)] $I_{\mathcal{GO}}^{N}$ satisfies (IB);
		\item [(v)] $I_{\mathcal{GO}}^{N}$ satisfies (CP) with respect to $N$;
		\item [(vi)] $I_{\mathcal{GO}}^{N}$ satisfies (R-CP) with respect to $N$.
	\end{enumerate}
\end{proposition}	

\begin{proof}
	Indeed,
	\begin{enumerate}
		\item [(i)] Directly from \cite[Prop. 6]{PBSS18b}, considering Remark \ref{rem-Crisp-GON-imp}.
		\item [(ii)] Since $N$ is crisp and 1 is a neutral element of $\mathcal{GO}$, it follows that:
		\begin{itemize}
			\item[(LOP)] For all $x,y \in [0,1]$ such that $x \leq y$, we have two situations: \begin{description} 
				
				\item [(1)]  If there exists $\alpha \in \left( 0,1 \right) $ such that $N = N_\alpha$, so, by Remark \ref{rem-Crisp-GON-imp} and \textbf{(C4)}, $I_{\mathcal{GO}}^{N}(x,y)=I^{\overline{\alpha}}_{\alpha}(x,y)$, then: 
				\begin{eqnarray} \label{Nalpha1}
				I_{\mathcal{GO}}^{N}(x,y)=\begin{cases}
				0, & \text{ if } x > \alpha \text{ and } y\leq \alpha. \\
				1, & \text{ if }  y > \alpha \text{ or } x \leq \alpha.
				\end{cases} 	
				\end{eqnarray}
				For $y\leq \alpha$, as $x \leq y$, it holds that $x \leq \alpha$. Hence one concludes that $I_{\mathcal{GO}}^{N}(x,y)=1$. For $y>\alpha$, it is immediate that $I_{\mathcal{GO}}^{N}(x,y)=1$.
				
				\item [(2)] If there exists $\alpha \in \left( 0,1 \right) $ such that $N = N^\alpha$, so, by Remark \ref{rem-Crisp-GON-imp} and  \textbf{(C3)}, $I_{\mathcal{GO}}^{N}(x,y)=I_{\underline{\alpha}}^{\alpha}(x,y)$, then:
				\begin{eqnarray} \label{Nalpha2}
				I_{\mathcal{GO}}^{N}(x,y)=\begin{cases}
				0, & \text{ if } x \geq \alpha \text{ and } y < \alpha.\\
				1, & \text{ if }  x < \alpha \text{ or } y \geq \alpha.
				\end{cases}	
				\end{eqnarray}
				For $y < \alpha$, as $x \leq y$, it holds that $x < \alpha$. So one concludes that $I_{\mathcal{GO}}^{N}(x,y)=1$. For $y\geq \alpha$, it is immediate that $I_{\mathcal{GO}}^{N}(x,y)=1$.
			\end{description}
			Therefore, it holds that $I_{\mathcal{GO}}^{N}$ satisfies $(LOP)$.

			\item[(ROP)] We also consider two situations: 
			\begin{description}
				
				\item [(1)]  If $N = N_\alpha$, for some $\alpha \in \left( 0,1 \right) $, then take $x,y \in [0,1]$ such that $ x > y > \alpha$. So, by Equation (\ref{Nalpha1}), $I_{\mathcal{GO}}^{N}(x,y) = 1$.

				\item [(2)]  If $N = N^\alpha$, for some $\alpha \in \left( 0,1 \right) $, then take $x,y \in [0,1]$ such that $ y < x < \alpha$. So, by Equation ~(\ref{Nalpha2}), $I_{\mathcal{GO}}^{N}(x,y) = 1$.
				
			\end{description}  In both situations, there exists $x > y$, but $I_{\mathcal{GO}}^{N}(x,y) = 1$, therefore $I_{\mathcal{GO}}^{N}$ does not satisfy $(ROP)$.
			
		\end{itemize}

		\item [(iii)] Given $x \in [0,1]$, since $N$ is crisp, $N(x) = 0$ or $N(x) = 1$. If $N(x) = 0$, then $I_{\mathcal{GO}}^{N}(x, x) = N(\mathcal{GO}(x, N(x))) = N(\mathcal{GO}(x, 0)) \stackrel{(\mathcal{GO}2)}{=} 1$. On the other hand, if $N(x) = 1$, then $I_{\mathcal{GO}}^{N}(x, x) = N(\mathcal{GO}(x, N(x))) =  N(\mathcal{GO}(x, 1)) = N(x) = 1$, since $1$ is the neutral element of $\mathcal{GO}$.
		
		\item [(iv)] Given $y \in [0,1]$, as $N$ is crisp, $N(y) = 0$ or $N(y) = 1$.
		
		\textbf{(1)} $N(y) = 0$: for all $x \in [0,1]$,	 
		\begin{eqnarray*}
			I_{\mathcal{GO}}^{N}(x, I_{\mathcal{GO}}^{N}(x,y)) \hspace{-2.0ex}&=&\hspace{-2.0ex} N(\mathcal{GO}(x, N(N(\mathcal{GO}(x, N(y)))))) \\
			&=& \hspace{-1.0ex} N(\mathcal{GO}(x, N(N(\mathcal{GO}(x, 0))))) \\
			&\stackrel{(\mathcal{GO}2)}{=}& \hspace{-1.0ex} N(\mathcal{GO}(x, N(N(0))))\\
			&=& \hspace{-1.0ex}N(\mathcal{GO}(x, 0)) \stackrel{(\mathcal{GO}2)}{=} 1 
		\end{eqnarray*}
		and 
		\begin{eqnarray*}
			I_{\mathcal{GO}}^{N}(x,y) = N(\mathcal{GO}(x, N(y))) = N(\mathcal{GO}(x, 0)) \stackrel{(\mathcal{GO}2)}{=} 1.
		\end{eqnarray*}
		
		\textbf{(2)} $N(y) = 1$: for all $x \in [0,1]$, since $1$ is the neutral element of $\mathcal{GO}$,
		\begin{eqnarray*}
			I_{\mathcal{GO}}^{N}(x, I_{\mathcal{GO}}^{N}(x,y)) &=& N(\mathcal{GO}(x, N(N(\mathcal{GO}(x, N(y))))))\\
			&=& N(\mathcal{GO}(x, N(N(\mathcal{GO}(x, 1)))))\\
			&=& N(\mathcal{GO}(x, N(N(x))))
		\end{eqnarray*}
		and $I_{\mathcal{GO}}^{N}(x,y) = N(\mathcal{GO}(x, N(y))) = N(\mathcal{GO}(x, 1)) = N(x)$. So, if $N(x) = 0$, then 
		\begin{eqnarray*}
			I_{\mathcal{GO}}^{N}(x, I_{\mathcal{GO}}^{N}(x,y))= N(\mathcal{GO}(x, 1)) = N(x) = 0
		\end{eqnarray*}
		and $I_{\mathcal{GO}}^{N}(x,y) = N(x) = 0$. Now, if $N(x) = 1$, then, by $(\mathcal{GO}2)$, $I_{\mathcal{GO}}^{N}(x, I_{\mathcal{GO}}^{N}(x,y)) = N(\mathcal{GO}(x, 0)) = 1$ and $I_{\mathcal{GO}}^{N}(x,y) = N(x) = 1$. Therefore, in any case, 
		\begin{eqnarray*}
			I_{\mathcal{GO}}^{N}(x, I_{\mathcal{GO}}^{N}(x,y)) = I_{\mathcal{GO}}^{N}(x,y).
		\end{eqnarray*}
		
		\item [(v)] Given $y \in [0,1]$, as $N$ is crisp, $N(y) = 0$ or $N(y) = 1$.
		
		\textbf{(1)} $N(y) = 0$: $I_{\mathcal{GO}}^{N}(x, y) = 
		N(\mathcal{GO}(x, 0)) \stackrel{(\mathcal{GO}2)}{=} 1$ and $I_{\mathcal{GO}}^{N}(N(y),N(x)) =  N(\mathcal{GO}(0, N(N(x)))) \stackrel{(\mathcal{GO}2)}{=} 1$, for all $x \in [0,1]$.
		
		\textbf{(2)} $N(y) = 1$: since $1$ is the neutral element of $\mathcal{GO}$, $I_{\mathcal{GO}}^{N}(x,y) = N(\mathcal{GO}(x, N(y))) = N(\mathcal{GO}(x, 1)) = N(x)$ and, we also have that
		\begin{eqnarray*}
			I_{\mathcal{GO}}^{N}(N(y), N(x))&=& N(\mathcal{GO}(N(y), N(N(x))))\\
			&=& N(\mathcal{GO}(1, N(N(x))))\\
			&=& N(N(N(x))),
		\end{eqnarray*}
		for all $x \in [0,1]$. Since $N$ is crisp, $N(N(N(x))) = N(x)$ for all $x \in [0,1]$. Therefore,  $I_{\mathcal{GO}}^{N}(N(y),N(x))= I_{\mathcal{GO}}^{N}(x,y)$.
		%
		
		\item [(vi)] Given $y \in [0,1]$, as $N$ is crisp, $N(y) = 0$ or $N(y) = 1$.
		
		\textbf{(1)} $N(y) = 0$: since $1$ is the neutral element of $\mathcal{GO}$, for all $x \in [0,1]$, 
		\begin{eqnarray*}
			I_{\mathcal{GO}}^{N}(x, N(y)) &=& N(\mathcal{GO}(x, N(0)))\\
			&=& N(\mathcal{GO}(x, 1)) = N(x)	
		\end{eqnarray*}
		and  $I_{\mathcal{GO}}^{N}(y,N(x)) = N(\mathcal{GO}(y, N(N(x))))$. If $N(x)$ $=0$ then, $I_{\mathcal{GO}}^{N}(x, N(y)) = 0 = N(y)= N(\mathcal{GO}(y,1))= I_{\mathcal{GO}}^{N}(y,N(x))$ and if  $N(x)=1$ then, $I_{\mathcal{GO}}^{N}(x, N(y)) = 1 = N(0)= N(\mathcal{GO}(y,0))= I_{\mathcal{GO}}^{N}(y,N(x))$.

		\textbf{(2)} $N(y) = 1$: since $1$ is the neutral element of $\mathcal{GO}$, $I_{\mathcal{GO}}^{N}(x,N(y)) = N(\mathcal{GO}(x, N(1))) \hspace{-0.5ex}=\hspace{-0.5ex} N(\mathcal{GO}(x, 0)) \hspace{-1.0ex} \stackrel{(\mathcal{GO}2)}{=} 1$ and $I_{\mathcal{GO}}^{N}(y, N(x))$ $ = N(\mathcal{GO}(y, N(N(x))))$, for all $x \in [0,1]$. So, if $N(x) = 0$, then 
		\begin{eqnarray*}
			I_{\mathcal{GO}}^{N}(y, N(x)) &=& N(\mathcal{GO}(y, N(0))) = N(\mathcal{GO}(y,1))\\
			&=& N(y) = 1.
		\end{eqnarray*}
		However, if $N(x) = 1$, then, by $(\mathcal{GO}2)$
		\begin{eqnarray*}
			I_{\mathcal{GO}}^{N}(y, N(x))\hspace{-0.3ex} = \hspace{-0.3ex} N(\mathcal{GO}(y, N(1))) \hspace{-0.3ex} =\hspace{-0.3ex} N(\mathcal{GO}(y, 0)) \hspace{-0.3ex} =\hspace{-0.3ex} 1.
		\end{eqnarray*}
		Therefore, in any case, $I_{\mathcal{GO}}^{N}(x, N(y)) = I_{\mathcal{GO}}^{N}(y, N(x))$.		   
	\end{enumerate}
\end{proof}	

	\subsection{Aggregating $(\mathcal{GO},N)$-Implications}
	
	In \cite{RBB13}, it was performed a study on $\mathcal{I}_A$ fuzzy implications obtained by the composition of an aggregation function $A$ and a family $\mathcal{I}$ of fuzzy implications. Here we verify under which conditions an $\mathcal{I}_A$-operator is a $(\mathcal{GO},N)$-implication, whenever $\mathcal{I}$ is a family of $(\mathcal{GO},N)$-implication functions.
	
	\begin{definition} \cite{RBB13} \label{def_AF-operator}
		Let $A:[0,1]^{n} \rightarrow [0,1]$ be an aggregation function and $\mathcal{F} = \{F_{i}: [0,1]^{k} \rightarrow [0,1] \mid i \in \{1, 2, \ldots, n\} \}$ be a family of $k$-ary functions. An $\mathbf{(A,\mathcal{F})}$\textbf{-operator} on $[0,1]$, denoted by $\mathcal{F}_A : [0,1]^{k} \rightarrow [0,1]$, is obtained as the composition $\mathcal{F}_A(x_1,\ldots,x_k)$, given by:
		\begin{eqnarray}\label{eq-F_A}
		A(F_1(x_1,\hspace{-0.3ex} \ldots\hspace{-0.3ex},x_k), F_2(x_1,\hspace{-0.3ex}\ldots \hspace{-0.3ex},x_k),\hspace{-0.3ex} \ldots\hspace{-0.3ex}, F_n(x_1,\hspace{-0.3ex} \ldots \hspace{-0.3ex},x_k)).
		\end{eqnarray}
	\end{definition}
	
	In \cite{RBB13}, it has been shown that $\mathcal{F}_A$ preserves some properties of $F_i$ for $i \in \{1, 2, \ldots, n\}$. For example, if $F_i$ are fuzzy implications then $\mathcal{F}_A$ is also a fuzzy implication.
	
	\begin{lemma}\label{lem-a-cont}
		Let $A: [0,1]^{n} \rightarrow [0,1]$ be an aggregation function and $\mathcal{GO}^* = \{{\mathcal{GO}}_i: [0,1]^{k} \rightarrow [0,1] \mid i \in \{1, 2, \ldots, n\} \}$ be a family of general overlap functions. Then $\mathcal{GO}_{A}^{*}$ is a general overlap function whenever $A$ is continuous.
	\end{lemma}
	
	\begin{proof}
		We will verify that $\mathcal{GO}_{A}^{*}$ satisfies the conditions that define a general overlap function:
		
		\begin{quote}	
			\begin{itemize}
				\item[($\mathcal{GO}$1)] Indeed, for all $x_1, \ldots, x_k \in [0,1]$, since $\mathcal{GO}_{i}$ is commutative for all $i \in \{1, \ldots, n\}$,
				\begin{eqnarray*}
					\mathcal{GO}_{A}^{*}\hspace{-1.0ex}&(&\hspace{-1.0ex}x_1, \ldots, x_r, \ldots, x_s, \ldots, x_k) =\\
					&=& A(\mathcal{GO}_{1}(x_1, \ldots, x_r, \ldots, x_s, \ldots, x_k), \ldots,\\
					&& \hspace{6.0ex} \mathcal{GO}_{n}(x_1, \ldots, x_r, \ldots, x_s, \ldots, x_k))\\
					&=& A(\mathcal{GO}_{1}(x_1, \ldots, x_s, \ldots, x_r, \ldots, x_k), \ldots,\\
					&& \hspace{6.0ex} \mathcal{GO}_{n}(x_1, \ldots, x_s, \ldots, x_r, \ldots, x_k))\\
					&=&\mathcal{GO}_{A}^{*}(x_1, \ldots, x_s, \ldots, x_r, \ldots, x_k),
				\end{eqnarray*}
				for any $r,s \in \{1, \ldots, k \}$.
				\item[($\mathcal{GO}$2)] If $\prod_{i=1}^k x_i = 0$, then, by ($\mathcal{GO}$2), $\mathcal{GO}_{i}(x_1, \ldots, x_k) = 0$ for all $i \in \{1, \ldots, n\}$, so
				\begin{eqnarray*}
					\mathcal{GO}_{A}^{*} \hspace{-1.3ex}&(& \hspace{-1.3ex} x_1, \ldots, x_k) =\\
					&=& \hspace{-0.7ex} A(\mathcal{GO}_{1}(x_1, \ldots, x_k), \ldots, \mathcal{GO}_{n}(x_1, \ldots, x_k))\\
					&=& \hspace{-0.7ex} A(0, \ldots, 0) \stackrel{(A1)}{=} 0.
				\end{eqnarray*}
				
				\item[($\mathcal{GO}$3)] If $\prod_{i=1}^k x_i = 1$, then, by ($\mathcal{GO}$3), $\mathcal{GO}_{i}(x_1, \ldots, x_k) = 1$ for all $i \in \{1, \ldots, n\}$, so
				\begin{eqnarray*}
					\mathcal{GO}_{A}^{*}\hspace{-1.3ex}&(& \hspace{-1.3ex}x_1, \ldots, x_k) =\\
					&=& \hspace{-0.7ex}  A(\mathcal{GO}_{1}(x_1, \ldots, x_k), \ldots, \mathcal{GO}_{n}(x_1, \ldots, x_k))\\
					&=& \hspace{-0.7ex} A(1, \ldots, 1)\stackrel{(A1)}{=} 1.
				\end{eqnarray*}
				
				\item[($\mathcal{GO}$4)] The result follows straightforward, since $A$ and $\mathcal{GO}_i$ are increasing, for all $i \in \{1, 2, \ldots, n\}$.
				
				\item[($\mathcal{GO}$5)] Since $A$ and $\mathcal{GO}_i$ are continuous, for all $i \in \{1, 2, \ldots, n\}$, the result follows straightforward.
			\end{itemize}	
		\end{quote}
		Therefore, $\mathcal{GO}_{A}^{*}$ is a general overlap function.
	\end{proof}
	
	\begin{proposition} \label{prop-GON-imp-agg}
		Let $A: [0,1]^{n} \rightarrow [0,1]$ be a continuous aggregation function and let $\mathcal{I} = \{I_{\mathcal{GO}_{i}}^{N_{i}}: [0,1]^{2} \rightarrow [0,1] \mid i \in \{1, \ldots, n\} \}$ be a family of $(\mathcal{GO},N)$-implications. Then,  $\mathcal{I}_A$ is a $(\mathcal{GO},N)$-implication whenever $N_{i} = N$ for $i \in \{1, 2, \ldots , n\}$ and $N$ is a strong negation.
		%
	\end{proposition}
	
	\begin{proof}
		Consider the family of $(\mathcal{GO},N)$-implications, $\mathcal{I} = \{I_{\mathcal{GO}_{i}}^{N_{i}}: [0,1]^{2} \rightarrow [0,1] \mid i \in \{1, 2, \ldots, n\} \}$. Then, since $N_{i} = N$ and $N$ is a strong negation, for all $0 \leq i \leq n$, 	$\mathcal{I}_{A}(x,y)=$
		\begin{eqnarray*}
		 &\stackrel{\text{Eq.} (\ref{eq-F_A})}{=}& \hspace{-2.3ex} A(I_{\mathcal{GO}_{1}}^{N_{1}}(x,y), \ldots, I_{\mathcal{GO}_{n}}^{N_{n}}(x,y))\\
			& \stackrel{\text{Eq.} (\ref{NossaIMP_FOverlap})}{=}& \hspace{-2.7ex} A(N_{1}(\mathcal{GO}_{1}(x,N_{1}(y))), \ldots, N_{n}(\mathcal{GO}_{n}(x,N_{n}(y))))\\
			&=& \hspace{-2.7ex}A(N(\mathcal{GO}_{1}(x,N(y))), \ldots, N(\mathcal{GO}_{n}(x,N(y))))\\
			& \stackrel{\text{Eq.} (\ref{eq-f_N}) \ \ref{N5}}{=}& \hspace{-0.5ex}  N(A_{N}(\mathcal{GO}_{1}(x,N(y)), \ldots, \mathcal{GO}_{n}(x,N(y))))\\
			& \stackrel{\text{Eq.} (\ref{eq-F_A})}{=}&\hspace{-2ex} N(\mathcal{GO}_{A_{N}}^{*}(x,N(y))) \\
			& \stackrel{\text{Eq.} (\ref{NossaIMP_FOverlap})}{=}&\hspace{-2.5ex} I_{\mathcal{GO}_{A_{N}}^{*}}^N(x,y).
		\end{eqnarray*}
		By Proposition \ref{AN_aggregation}, $A_{N}$ is an aggregation function. Besides, by the continuity of $A$ and $N$, we have that $A_{N}$ is continuous. So, by Lemma \ref{lem-a-cont}, $\mathcal{GO}_{A_{N}}^{*}$ is a general overlap function. Therefore, since $\mathcal{I}_{A} = I_{\mathcal{GO}_{A_{N}}^{*}}^N$, then $\mathcal{I}_{A}$ is a $(\mathcal{GO},N)$-implication function.
	\end{proof}

	\begin{corollary}
		Let $A: [0,1]^{n} \rightarrow [0,1]$ be a continuous aggregation function and let $\mathcal{I} = \{I_{\mathcal{GO}_{i}}^{N_{i}}: [0,1]^{2} \rightarrow [0,1] \mid i \in \{1, 2, \ldots, n\} \}$, for $i \in \{1, 2, \ldots , n\}$, be a family of $(\mathcal{GO},N)$-implications. If $N$ is a strong negation, then for $\mathcal{I}_{A}$ with $N_{i} = N$ for $i \in \{1, 2, \ldots , n\}$, it holds that:
		
		\begin{enumerate}
			\item [(i)] 
			$\mathcal{I}_{A}$
			satisfies L-CP(N);
			\item [(ii)] If $N$ is also strict, then
			$\mathcal{I}_{A}$ satisfies R-CP(N$^{-1}$).
			\item [(iii)] $\mathcal{I}_{A}$ satisfies CP(N).
		\end{enumerate}
	\end{corollary}

	\begin{proof}
		Straightforward from Propositions \ref{Prop.(O,N)} and \ref{prop-GON-imp-agg}.
	\end{proof}

\section{Intersections between Families of Fuzzy Implications}

In this section we present results regarding the intersections that exist among the families of fuzzy implications $(\mathcal{GO},N)$, $(G,N)$, $QL$, $R_{O}$ and $D$-implications derived from (general) overlap and grouping functions $O$ and $G$, respectively, and fuzzy negations $N$. We will represent these families by $\mathbb{I}_{\mathcal{G}\mathbb{O}}^{\mathbb{N}}, \mathbb{I}_{\mathbb{G,N}}, \mathbb{I}_{\mathbb{O,G,N}}$,  $\mathbb{I}_{\mathbb{O}}$ and $\mathbb{I}_{\mathbb{D}}$, respectively. 

\subsection{Intersections between $(\mathcal{GO},N)$ and $(G,N)$-implications}

\begin{proposition}
	Let $N$ and $N'$ be fuzzy negations, $\mathcal{GO}$ be a bivariate general overlap function and $G$ be a grouping function such that $I_{\mathcal{GO}}^{N} = I_{G, N'}$.
	\begin{itemize}
		\item [(i)] If $N$ is strict and $N'$ is frontier, then $\mathcal{GO}$ is an overlap function.
		\item [(ii)] If  $1$ is the neutral element of $\mathcal{GO}$, then:
		\begin{enumerate}
			\item [(a)] If $N$ is a strong negation, then $N = N'$;
			\item [(b)] If $N$ is continuous and $N = N'$, then $N$ is strong;
			\item [(c)] $N$ is strong if and only if $0$ is the neutral element of $G$.
		\end{enumerate}
		\item [(iii)] If  $0$ is the neutral element of $G$, then:
		\begin{enumerate}
			\item [(a)] $N'$ is strong if and only if $N' = N $;
			\item [(b)] $N'$ is strong if and only if $1$ is the neutral element of $\mathcal{GO}$.
		\end{enumerate}
	\end{itemize}
\end{proposition}

\begin{proof}
	\begin{itemize}
		\item [(i)] Indeed, if $\mathcal{GO}(x,y) = 0$, then
		\begin{eqnarray*}
			N(\mathcal{GO}\hspace{-2.0ex}&(&\hspace{-2.0ex}x,y)) = 1 \Rightarrow \\
			&\Rightarrow&  I_{G, N}(x,N^{-1}(y)) = I_{\mathcal{GO}}^{N} (x, N^{-1}(y))= 1\\
			&\Rightarrow& G(N(x), N^{-1}(y)) = 1\\
			&\stackrel{\ref{G3}}{\Rightarrow}& N(x) =1 \text{ or } N^{-1}(y) = 1\\
			&\stackrel{\ref{G3}}{\Rightarrow}& x = 0 \text{ or } y = 0.
		\end{eqnarray*}
		And, if $\mathcal{GO}(x,y) = 1$, then
		\begin{eqnarray*}
			N(\mathcal{GO}\hspace{-2.0ex}&(&\hspace{-2.0ex}x,y)) = 0 \Rightarrow \\
			&\Rightarrow& I_{G, N}(x,N^{-1}(y)) = I_{\mathcal{GO}}^{N} (x, N^{-1}(y))= 0\\
			&\Rightarrow& G(N(x), N^{-1}(y)) = 0\\
			&\stackrel{\ref{G2}}{\Rightarrow}& N(x) = 0 \text{ and } N^{-1}(y) = 0\\
			&\stackrel{\ref{G2}}{\Rightarrow}& x = 1 \text{ and } y = 1.
		\end{eqnarray*}
		So, $\mathcal{GO}$ satisfies \ref{O2} and \ref{O3}. We conclude that $\mathcal{GO}$ is a overlap function. 
		\item [(ii)] Indeed,
		\begin{enumerate}
			\item [(a)] by Proposition 3.4$(xxi)$ in \cite{DBS14} we have that $I_{G, N'}$ satisfies $R-CP(N')$, so
			\begin{eqnarray*}
				N(y)\hspace{-1.5ex} &=& \hspace{-1.5ex} I_{\mathcal{GO}}^{N}(y,0)= I_{G, N'}(y,0)\\
				& \stackrel{\text{R-CP(N')}}{=}& \hspace{-1.0ex} I_{G, N'}(1,N'(y)) = I_{\mathcal{GO}}^{N}(1,N'(y))\\
				&=& \hspace{-1.5ex} N(\mathcal{GO}(1, N(N'(y)))) \hspace{-0.5ex} \stackrel{(N5)}{=}\hspace{-0.5ex} N'(y),
			\end{eqnarray*}
			for all $y \in [0,1]$. Therefore, $N=N'$.
			
			\item [(b)] Since $I_{G, N'}$ satisfies $R-CP(N')$ and $I_{\mathcal{GO}}^{N} = I_{G, N'}$, $I_{\mathcal{GO}}^{N}(x,N'(y)) = I_{\mathcal{GO}}^{N}(y,N'(x))$. So, for $x=1$, $I_{\mathcal{GO}}^{N}(1,N'(y)) = I_{\mathcal{GO}}^{N}(y,N'(1))$, i.e., 
			\begin{eqnarray*}
				N(\mathcal{GO}(1, N(N'(y)))) = N(\mathcal{GO}(y, N(N'(1)))).
			\end{eqnarray*}
			Since $1$ is the neutral element of $\mathcal{GO}$ and $N=N'$, for all $y \in [0,1]$, $N(N(N(y))) = N(y)$. Now, since $N$ is continuous, for every $x \in [0,1]$, there is $y \in [0,1]$ such that $x = N(y)$. So, $N(N(x)) = x$, for all $x \in [0,1]$.
			
			\item [(c)] For all $y \in [0,1]$, $N(N(y)) = N(\mathcal{GO}(1, N(y))) = I_{\mathcal{GO}}^{N}(1,y) = I_{G, N'}(1,y) = G(N'(1), y)$ $= G(0,y)$. So the result holds.
		\end{enumerate}
		\item [(iii)] Indeed,
		\begin{enumerate}
			\item [(a)] by Proposition 3.4(ii) in \cite{DBS14} we have that $I_{G, N'}$ satisfies $(NP)$, so
			\begin{eqnarray*}
				y &\stackrel{\text{(NP)}}{=}& I_{G, N'}(1,y) = I_{\mathcal{GO}}^{N}(N(0),y)\\
				& \stackrel{\text{Prop. \ref{Prop.(O,N)}(i)}}{=}& I_{\mathcal{GO}}^{N}(N(y),0)
				= I_{G, N'}(N(y),0)\\
				&=&  N'(N(y)),
			\end{eqnarray*}
			for all $y \in [0,1]$. Therefore, the results follows.
			\item [(b)] Consider $N'$ as a strong negation, then by the previous item, $N'=N$. So,
			\begin{eqnarray*}
				x &=& N'(N'(x)) = N'(G(N'(x),0)) \\
				&=& N'(I_{G, N'}(x,0)) = N'(I_{\mathcal{GO}}^{N}(x,0))\\
				&=& N'(N(\mathcal{GO}(x, N(0)))) \stackrel{\text{N'=N}}{=} \mathcal{GO}(x,1),	
			\end{eqnarray*}
			for all $x \in [0,1]$. Therefore, $1$ is the neutral element of $\mathcal{GO}$. Conversely, $N(x) = N(\mathcal{GO}(x, N(0))) = I_{\mathcal{GO}}^{N}(x,0)= I_{G, N'}(x,0) =N'(x)$,
			and therefore by $(a)$ of item $(ii)$, $N'$ is a strong negation.
		\end{enumerate}
	\end{itemize}
\end{proof}

The next propositions show that strict $(\mathcal{GO},N)$-impli-\ cation functions generated by general overlap functions satisfying ($\mathcal{GO}$2a) and ($\mathcal{GO}$3a) are strict $(G,N)$-implication functions and vice-versa.

\begin{proposition} \label{prop-eq-gn-impl}  Let $N$ be a strict fuzzy negation, $\mathcal{GO}$ be a general overlap function satisfying ($\mathcal{GO}$2a) and ($\mathcal{GO}$3a), and  $G$ be the grouping function defined  in Eq. (\ref{Group_of_GOver}). Then, $I_{\mathcal{GO}}^{N}=I_{G,N^{-1}}$.
\end{proposition}
\begin{proof} For all $x,y \in [0,1]$, since $N$ is strict, it follows that:
	\begin{eqnarray*}
		I_{\mathcal{GO}}^{N}(x,y) &=&  N (\mathcal{GO}(x, N(y)))  \\ &=& N (\mathcal{GO}(N(N^{-1}(x)), N(y)))\\
		&\stackrel{\text{Eq.}(\ref{Group_of_GOver})}{=}& G(N^{-1}(x), y)   \stackrel{\text{Eq.}(\ref{eq-GN-impl})}=  I_{G,N^{-1}} (x,y).
	\end{eqnarray*}
	
\end{proof}

\begin{proposition} \label{prop-eq-gn-impl2}  Let $N$ be a strict negation, $G$ be a grouping function and  $\mathcal{GO}$ be the general overlap function  defined  in Equation (\ref{Over_of_Group}). Then, $I_{G,N}=I_{\mathcal{GO}}^{N^{-1}}$.
\end{proposition}
\begin{proof} For all $x,y \in [0,1]$, since $N$ is strict, it follows that:
	\begin{eqnarray*}
		I_{G,N} (x,y) &\stackrel{\text{Eq.}(\ref{eq-GN-impl})}{=}&  N^{-1}(N(G(N(x), N(N^{-1}(y)))))\\
		&\stackrel{\text{Eq.}(\ref{Over_of_Group})}{=}& N^{-1} (\mathcal{GO}(x, N^{-1}(y)) \stackrel{\text{Eq.}(\ref{NossaIMP_FOverlap})}=  I_{\mathcal{GO}}^{N^{-1}}(x,y).
	\end{eqnarray*}
\end{proof}


\begin{corollary} \label{dualidade(O,N)(G,N)}
	Let $I$ be a fuzzy implication. Then, $I$ is a strict $(\mathcal{GO},N)$-implication with $\mathcal{GO}$ satisfying conditions ($\mathcal{GO}$2a) and ($\mathcal{GO}$3a) if and only if $I$ is a strict $(G,N)$-implication.
\end{corollary}

\begin{proof} Straightforward from Propositions \ref{prop-eq-gn-impl} and \ref{prop-eq-gn-impl2}.
\end{proof}

By Corollary \ref{dualidade(O,N)(G,N)} we have that the intersection of $(\mathcal{GO},N)$ and $(G,N)$-implications is non-empty: $\mathbb{I}_{\mathcal{G}\mathbb{O}}^{\mathbb{N}} \cap \mathbb{I}_{\mathbb{G,N}} \neq \emptyset$.
In addition, we also conclude that $\mathbb{I}_{\mathcal{G}\mathbb{O}}^{\mathbb{N^{*}}} = \mathbb{I}_{\mathbb{G,N^{*}}} \subseteq \mathbb{I}_{\mathcal{G}\mathbb{O}}^{\mathbb{N}} \cap \mathbb{I}_{\mathbb{G,N}}$, where $\mathbb{I}_{\mathcal{G}\mathbb{O}}^{\mathbb{N^{*}}}$ is the family of all strict $(\mathcal{GO},N)$-implications and, analogously,  $\mathbb{I}_{\mathbb{G,N^{*}}}$ is the family of all strict $(G,N)$-implications.  


\begin{proposition}\label{prop-GON-GN-RanI}
	Let $I \in \mathcal{FI}$  such that $Ran(I)\neq [0,1]$. If $I$ is a $(\mathcal{GO},N)$-implication then $I$ is not a $(G,N)$-implication.
\end{proposition}
\begin{proof} Suppose that $I$ is a $(G,N)$-implication. Then, there is a grouping $G$ and a fuzzy negation $N$ such that $I(x,y)=G(N(x),y)$ for each $x,y\in [0,1]$. However, since $G$ is continuous and $G(N(0),0)=0$, $G(N(0),1)=1$, then for any $y\in [0,1]$ there exists $x\in [0,1]$ such that $I(0,x) = G(1,x)=y$. Therefore, $Ran(I)=[0,1]$.
\end{proof}

\begin{corollary}
	Each crisp $(\mathcal{GO},N)$-implication is not a $(G,N)$-implication.
\end{corollary}

Let $\mathbb{I}_{\mathcal{G}\mathbb{O}}^{\mathbb{N}}=\{I\in \mathbb{I}_{\mathcal{G}\mathbb{O}}^{\mathbb{N}} | Ran(I)\neq [0,1]\}$.
Proposition \ref{prop-GON-GN-RanI} proves that $\mathbb{I}_{\mathcal{G}\mathbb{O}}^{\mathbb{N}} \cap \mathbb{I}_{\mathbb{G,N}}=\emptyset$.

%

Thus, there are $(\mathcal{GO},N)$-implications that are not $(G,N)$-implications and therefore,
the class of $(\mathcal{GO}, N)$-implications 
is not contained in the class of $(G,N)$-implications. But the converse also holds as shown in the next proposition.

\begin{proposition}
	There are $(G,N)$-implications that are not $(\mathcal{GO},N)$-implications.
\end{proposition}

\begin{proof}
	Take the $(G,N)$-implication $I_{G, N}$, where  $G(x,y) = \max( x,y)$ and $N = N_{\top}$. Thus,
	\begin{eqnarray*}
		I_{G, N}(x,y) &=& \max(N_{\top}(x), y) =  \begin{cases}
			\max(0,y), & \text{ if }  x = 1 \\
			1, & \text{ if }  x < 1
		\end{cases} \\
		&=& \begin{cases}
		y, & \text{ if }  x = 1 \\
		1, & \text{ if }  x < 1
	\end{cases}.
\end{eqnarray*}
Suppose there exists a general overlap function $\mathcal{GO}$ and a fuzzy negation $N$ such that $$I_{\mathcal{GO}}^{N}(x,y) = \begin{cases}
y, & \text{ if }  x = 1 \\
1, & \text{ if }  x < 1.
\end{cases}$$ Thus, for $x = 1$, $I_{\mathcal{GO}}^{N}(1,y) = y$, for all $y \in [0,1]$,
\begin{eqnarray}\label{cond.1}
N(\mathcal{GO}(1,N(y))) = y
\end{eqnarray}
And, for $x < 1$, $I_{\mathcal{GO}}^{N}(x,y) = 1$, for all $y \in [0,1]$. So, in particular, for $y = 0$, since $\mathcal{GO}$ is commutative,
\begin{eqnarray}\label{cond.2}
N(\mathcal{GO}(1,x)) = 1
\end{eqnarray}
for all $x < 1$.  Now, given $y \in (0,1)$ we have that $N(y) = 1$ or $N(y) < 1$. If $N(y) = 1$ then, by $(\ref{cond.1})$ and $(\mathcal{GO}3)$, $y = N(\mathcal{GO}(1,N(y))) = N(\mathcal{GO}(1, 1)) = N(1) = 0$, which is a contradiction, since $y \in (0,1)$. And, if $N(y) < 1$ then $y \stackrel{\text{Eq.}(\ref{cond.1})}{=} N(\mathcal{GO}(1,N(y))) \stackrel{\text{Eq.}(\ref{cond.2})}{=} 1,$ which is a contradiction, since $y \in (0,1)$. In both cases we have a contradiction, so $I_{\max, N_{\top}}$ is not an $(\mathcal{GO},N)$ -implication.
\end{proof}

The last results ensure that $\mathbb{I}_{\mathcal{G}\mathbb{O}}^{\mathbb{N}} \nsubseteq \mathbb{I}_{\mathbb{G,N}}$ and $\mathbb{I}_{\mathbb{G,N}} \nsubseteq \mathbb{I}_{\mathcal{G}\mathbb{O}}^{\mathbb{N}}$. 

\subsection{Intersections between $(\mathcal{GO},N)$ and $QL$-implications}

Recall that a $QL$-operator built from a tuple $( O ,G ,N )$, where $O$ is an overlap function, $G$ is a grouping function and $N$ is a fuzzy negation, is an implication function if and only if $N=N_\top$, as seen in  \cite{Dimuro2017-ql}. Then, we conclude that

\begin{proposition} \label{rel.QL_ON}
	There are no fuzzy implications that are simultaneously $QL$ implication functions and $(\mathcal{GO},N)$-impli- \ cation functions.
\end{proposition}

\begin{proof}
	Indeed, by Proposition \ref{Prop.(O,N)}(i) any $(\mathcal{GO},N)$-implication function $I_{\mathcal{GO}}^{N}$ satisfies $L-CP(N)$, moreover by Theorem 3.1(v) in \cite{Dimuro2017-ql}, any $QL$-implication $I_{O,G,N_{\top}}$ does not satisfy $L-CP$ for any negation $N$.	
\end{proof}

\begin{corollary}
	There is no fuzzy implication function which is simultaneously a $QL$-implication function and a strict $(G,N)$-implication function.
\end{corollary}

\begin{proof}Straightforward from Corollary \ref{dualidade(O,N)(G,N)} and Proposition \ref{rel.QL_ON}.
\end{proof}

Therefore, one can conclude that the intersection of $QL$- and $(\mathcal{GO},N)$-implications is empty, i.e.
$\mathbb{I}_{\mathcal{G}\mathbb{O}}^{\mathbb{N}} \cap \mathbb{I}_{\mathbb{O,G,N}} = \emptyset$.
As a consequence, the intersection of  $QL$-implication functions and $(G,N)$-implication functions with $N$ being a strict negation, is also empty:
$\mathbb{I}_{\mathbb{O,G,N}} \cap \mathbb{I}_{\mathbb{G,N^{*}}} = \emptyset$.

\subsection{Intersections between $(\mathcal{GO},N)$ and $R_{O}$-implications}

\begin{proposition} \label{rel.RO_ON}
	There are no fuzzy implication functions  that are simultaneously $R_O$ and $(\mathcal{GO},N)$-implications.
\end{proposition}

\begin{proof}
	Indeed, by Proposition \ref{Prop.(O,N)}(i), any $(\mathcal{GO},N)$-implication $I_{\mathcal{GO}}^{N}$ satisfies $L-CP(N)$, however  in \cite[Theorem 4.2]{Dimuro2015} we see that every $R_O$-implication, $I_{O}$, does not satisfy $L-CP$ for any negation $N$.	
\end{proof}

Therefore, one can conclude that $(\mathcal{GO},N)$- and $R_{O}$-impli- \ cations do not intercept, i.e. $\mathbb{I}_{\mathcal{G}\mathbb{O}}^{\mathbb{N}} \cap \mathbb{I}_{\mathbb{O}} = \emptyset$.

\begin{figure}
	\begin{center}
		\caption{Intersections between families of fuzzy implication functions.}
		\label{fig-intersections}
		\begin{minipage}{1.3\textwidth}
			{\scriptsize	\begin{itemize}[label=$\ast$,leftmargin=1cm]
					\item $\mathbb{I}_{\mathcal{G}\mathbb{O}}^{\mathbb{N}}$ is a $(\mathcal{GO},N)$-implication, \qquad 	$\ast$ $\mathbb{I}_{\mathbb{O}}$ is an $R_{O}$-implication, 
					\item  $\mathbb{I}_{\mathbb{G,N}}$ is a $(G,N)$-implication, 	\qquad \;  $\ast$ $\mathbb{I}_{\mathbb{D}}$ is a $D$-implication.
					\item $\mathbb{I}_{\mathbb{O,G,N}}$ is a $QL$-implication,  		
				\end{itemize} }
			\end{minipage}
			\begin{tikzpicture}
			\draw \secondcircle node [above=1.3cm,label={[above=0.3cm]}] {$\mathbb{I}_{\mathbb{G},\mathbb{N}}$};
			
			\draw(0:4cm) node[ellipse,minimum height=2.3cm,minimum width=2.3cm,label={[above=1.8cm]270:$\mathbb{I}_{\mathbb{O},\mathbb{G},\mathbb{N}}$},draw] {};
			
			\draw \fourcircle node [] {$\mathbb{I}_{\mathbb{D}}$};

			\draw (1cm,0) node[ellipse,minimum height=2.5cm,minimum width=2.5cm,label={[above=1.8cm]270:$\mathbb{I}_{\mathcal{G}\mathbb{O}}^{\mathbb{N}}$},draw] {};
			
			\draw \thirdcircle node [] {$\mathbb{I}_{\mathbb{O}}$};
			
			%
			%
			\end{tikzpicture}
		\end{center}
	\end{figure}
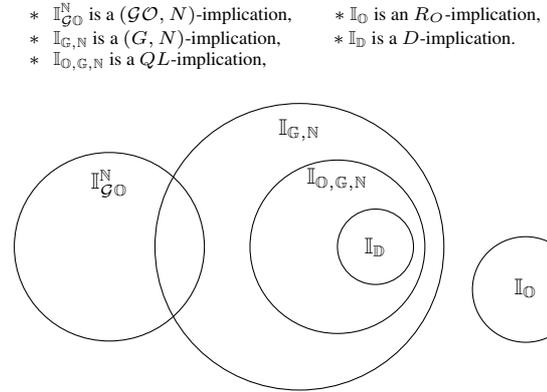

	\subsection{Intersections between $(\mathcal{GO},N)$ and $D$-implications}
	
	From the results given in \cite[Theorem 4.1]{d-impl} we know that every $D$-implication function is a $QL$-operation considering the greatest fuzzy negation.  Still, from \cite[Theorem 4.2]{d-impl} we know that every $D$-implication is a $(G,N)$-implication considering the greatest fuzzy negation. Therefore, it is straightforward that there are no intersections between $(\mathcal{GO},N)$ implication functions and $D$-implication / $(G,N)$-implication functions. Moreover, from \cite[Theorem 4.3]{d-impl} one can say that there is no intersection between $(\mathcal{GO},N)$ implication functions and $D$-implications.
	
	In Figure \ref{fig-intersections}, we illustrate the main results presented in this section. Note that the intersections between the families of $(G,N)$, $QL$, $R_O$ and $D$-implications had already been presented in \cite{Dimuro2015,DBS14,Dimuro2017-ql,dimuro2019law,d-impl}.

\section{Final Remarks}

In propositional logics, it is only necessary to consider the negation ($\neg$) and other  logical connective as primitive, either the implication ($\rightarrow$), the disjunction ($\vee$) or the conjunction ($\wedge$), since the other connectives can be defined in a standard form using only two primitive connectives, \cite{Mendelson15}. In particular, when the primitive connectives are the negation and the disjunction, the standard definition of the implication is given by Equation (\ref{eq-taut3a}) and when the primitive connectives are the negation and the conjunction, the standard definition of the implication is given by Equation (\ref{eq-taut3}). The first one, in fuzzy logics, had motived the introduction of several classes of fuzzy implication functions, such as the  $(S,N)$, $(G,N)$ and $(A,N)$ implications, where the disjunction is given, respectively, by a t-conorm $S$, a grouping function $G$ or a disjunctive aggregation function $A$ (e.g. see \cite{BJ07,DBS14,Pradera16}).  The second one, yielded implication functions defined by means of t-norms, e.g. see \cite{Bedregal07,PBSS18}. In this work we introduced a class of implication function based on this second standard definition  of the implication, where the conjunction is given by  generalized overlap functions.

The main contributions of this work are the investigation of properties satisfied by such implication functions and a study of  the intersections between them and other classes of implication functions derived from overlap and grouping functions provided. The summary of these intersections is illustrated in Figure \ref{fig-intersections}. Actually, we complete this study by also considering the class of $(T,N)$-implication functions, denoted by $\mathbb{I}_{\mathbb{T}}^{\mathbb{N}}$, which is also based on the standard definition of the implication given by Equation (\ref{eq-taut3}), but using a t-norm instead of a  general overlap function. Since each continuous t-norm is a general overlap function but the converse does not hold, then trivially we have that: $\mathbb{I}_{\mathcal{G}\mathbb{O}}^{\mathbb{N}}\cap \mathbb{I}_{\mathbb{T}}^{\mathbb{N}}\neq \emptyset$, $\mathbb{I}_{\mathcal{G}\mathbb{O}}^{\mathbb{N}}- \mathbb{I}_{\mathbb{T}}^{\mathbb{N}}\neq \emptyset$ and $\mathbb{I}_{\mathbb{T}}^{\mathbb{N}}- \mathbb{I}_{\mathcal{G}\mathbb{O}}^{\mathbb{N}}\neq \emptyset$. In addition, Table \ref{tab-comparison} shows some of the properties satisfied by the  
$(\mathcal{GO},N)$-implication functions and $(T,N)$-implication functions whenever we take into account: any fuzzy negation $N$, strong fuzzy negations (represented by $\mathbb{N}^*$), non-strong fuzzy negations (represented by $\mathbb{N}^+$) or crisp negations (represented by $\mathbb{N}_c$). For each property, yes/no means that the property is/is not held for each implication of that class. Additional restrictions may appear as footnotes like: no\footnote{If $N$ is strict.}, yes\footnote{1 is the neutral element of $\mathcal{GO}$.} or yes\footnote{$\mathcal{GO}$ is associative.}. Empty table cells mean that some implication functions of the class satisfy the property  whereas others do not. We can notice that indeed $\mathcal{GO}$-implication functions are more general since more properties are verified.

\begin{table}[h!]
	\caption{Some properties of fuzzy implication functions.}
	\label{tab-comparison}
	\begin{center}
		\begin{tabular}{|c|c|c|c|c|c|c|c|c|}
			\hline 
			Property  & $\mathbb{I}_{\mathbb{T}}^{\mathbb{N}^*}=\mathbb{I}_{\mathbb{S,N}^*}$ & $\mathbb{I}_{\mathbb{T}}^{\mathbb{N}^{+}}$ & $\mathbb{I}_{\mathbb{T}}^{\mathbb{N}_c}$  & $\mathbb{I}_{\mathcal{G}\mathbb{O}}^{\mathbb{N}^*}$ & $\mathbb{I}_{\mathcal{G}\mathbb{O}}^{\mathbb{N}_c}$  
			\\  
			\hline 
			EP  & yes	& no$^1$ & yes & yes$^3$ &yes$^2$  \\ 
			\hline 
			NP 	& yes & no & no &  yes$^2$ & no$^2$  \\ 
			\hline 
			ROP 	& &  & no & & no$^2$ \\ 
			\hline 
			LOP	 & &  & yes &  &yes$^2$  \\ 
			\hline 
			CP(N)  & yes	& no$^1$ & yes &  yes& yes$^2$   \\ 
			\hline 
			L-CP(N)  & yes & yes& yes &  yes & yes  \\ 
			\hline 
			R-CP(N) & yes	& no$^1$ & yes &  yes &yes$^2$  \\ 
			\hline 
		\end{tabular} 
	\end{center}
\end{table}

Our future works include studying the use of $\mathcal{GO}$ operators on other classes of implication functions such as the ones given in \cite{Vania-Fuzz18} and  the construction of  other classes of fuzzy subsethood measures like it was made in \cite{Dimuro2017-ql,PBSS18}, which can be used to generate fuzzy entropies, similarity measures and penalty functions, as seen in \cite{Helida2019}, and applied in many ways.

\section*{Acknowledgments}
This study was funded by the Brazilian funding agencies: CNPq (307781 / 2016-0, 305882 / 2016-3, 301618 / 2019-4), and FAPERGS (19 / 2551-0001660-3, 19 / 2551-0001279-9); and the Spanish Ministry of Science and Technology (TIN2016- 77356-P (AEI/FEDER, UE)).

\bibliography{mybiblio}   

\begin{thebibliography}{56}
\providecommand{\natexlab}[1]{#1}
\providecommand{\url}[1]{{#1}}
\providecommand{\urlprefix}{URL }
\expandafter\ifx\csname urlstyle\endcsname\relax
  \providecommand{\doi}[1]{DOI~\discretionary{}{}{}#1}\else
  \providecommand{\doi}{DOI~\discretionary{}{}{}\begingroup
  \urlstyle{rm}\Url}\fi
\providecommand{\eprint}[2][]{\url{#2}}

\bibitem[{Baczy{\'{n}}ski(2013)}]{Bac13}
Baczy{\'{n}}ski M (2013) On the applications of fuzzy implication functions.
  In: Balas VE, Fodor J, V{\'a}rkonyi-K{\'o}czy AR, Dombi J, Jain LC (eds) Soft
  Computing Applications, Springer Berlin Heidelberg, Berlin, Heidelberg, pp
  9--10

\bibitem[{Baczy\'nski and Jayaram(2007)}]{BJ07}
Baczy\'nski M, Jayaram B (2007) On the characterizations of
  {(S,N)}-implications. Fuzzy Sets and Systems 158(15):1713--1727

\bibitem[{Baczy\'nski and Jayaram(2008)}]{BJ08}
Baczy\'nski M, Jayaram B (2008) Fuzzy Implications, Studies in Fuzziness and
  Soft Computing, vol 231. Springer

\bibitem[{Baczy\'nski et~al.(2013)Baczy\'nski, Beliakov, Bustince, and
  Pradera}]{BBBP13}
Baczy\'nski M, Beliakov G, Bustince H, Pradera A (2013) Advances in Fuzzy
  Implications Functions, Studies in Fuzziness and Soft Computing, vol 300.
  Springer-Verlag Berlin Heidelberg

\bibitem[{Baczy\'nski et~al.(2015)Baczy\'nski, Jayaram, Massanet, and
  Torrens}]{BJ2015}
Baczy\'nski M, Jayaram B, Massanet S, Torrens J (2015) Fuzzy implications:
  Past, present, and future. In: Kacprzyk J, Pedrycz W (eds) Springer Handbook
  of Computational Intelligence, Springer Berlin Heidelberg, Berlin,
  Heidelberg, pp 183--202

\bibitem[{Bedregal(2007)}]{Bedregal07}
Bedregal BC (2007) A normal form which preserves tautologies and contradictions
  in a class of fuzzy logics. Journal of Algorithms 62(3):135 -- 147

\bibitem[{Bedregal(2010)}]{Bedregal08}
Bedregal BC (2010) On interval fuzzy negations. Fuzzy Sets and Systems
  161(17):2290--2313

\bibitem[{Bedregal et~al.(2013)Bedregal, Dimuro, Bustince, and
  Barrenechea}]{BDBB13}
Bedregal BC, Dimuro GP, Bustince H, Barrenechea E (2013) New results on overlap
  and grouping functions. Information Sciences 249:148--170

\bibitem[{Beliakov et~al.(2007)Beliakov, Pradera, and
  Calvo}]{Beliakov:Aggregations}
Beliakov G, Pradera A, Calvo T (2007) Aggregation Functions: A Guide for
  Practitioners. Springer, Berlin

\bibitem[{Bloch(2009)}]{Bloch09}
Bloch I (2009) Duality vs. adjunction for fuzzy mathematical morphology and
  general form of fuzzy erosions and dilations. Fuzzy Sets and Systems
  160(13):1858 -- 1867

\bibitem[{Bustince et~al.(2010)Bustince, Fernandez, Mesiar, Montero, and
  Orduna}]{Bus10a}
Bustince H, Fernandez J, Mesiar R, Montero J, Orduna R (2010) Overlap
  functions. Nonlinear Analysis: Theory, Methods \& Applications
  72(3-4):1488--1499

\bibitem[{Bustince et~al.(2012)Bustince, Pagola, Mesiar, Hullermeier, and
  Herrera}]{BPMHF12}
Bustince H, Pagola M, Mesiar R, Hullermeier E, Herrera F (2012) Grouping,
  overlap, and generalized bientropic functions for fuzzy modeling of pairwise
  comparisons. IEEE Transactions on Fuzzy Systems 20(3):405--415

\bibitem[{Bustince et~al.(2013)Bustince, Fernandez, Sanz, Baczy\'nski, and
  Mesiar}]{BFSBM13}
Bustince H, Fernandez J, Sanz J, Baczy\'nski M, Mesiar R (2013) Construction of
  strong equality index from implication operators. Fuzzy Sets and Systems
  211:15--33

\bibitem[{Carbonell and Torrens(2010)}]{carbonell10}
Carbonell M, Torrens J (2010) Continuous {R}-implications generated from
  representable aggregation functions. Fuzzy Sets and Systems 161(17):2276 --
  2289, theme: Aggregations, Connectives and Non-Additive Measures

\bibitem[{Cruz et~al.(2018)Cruz, Bedregal, and Santiago}]{CBS18}
Cruz A, Bedregal B, Santiago R (2018) On the characterizations of fuzzy
  implications satisfying {I(x,I(y,z))=I(I(x,y),I(x,z))}. International Journal
  of Approximate Reasoning 93:261 -- 276

\bibitem[{De~Miguel et~al.(2019)De~Miguel, G\'omez, Rodr\'iguez, Montero,
  Bustince, Dimuro, and Sanz}]{DeMiguel2019}
De~Miguel L, G\'omez D, Rodr\'iguez JT, Montero J, Bustince H, Dimuro GP, Sanz
  JA (2019) General overlap functions. Fuzzy Sets and Systems 372:81--96

\bibitem[{Dimuro and Bedregal(2014)}]{Dimuro201439}
Dimuro GP, Bedregal B (2014) Archimedean overlap functions: The ordinal sum and
  the cancellation, idempotency and limiting properties. Fuzzy Sets and Systems
  252:39 -- 54

\bibitem[{Dimuro and Bedregal(2015)}]{Dimuro2015}
Dimuro GP, Bedregal B (2015) On residual implications derived from overlap
  functions. Information Sciences 312:78 -- 88

\bibitem[{Dimuro et~al.(2014{\natexlab{a}})Dimuro, Bedregal, Bustince, Mesiar,
  and Asiain}]{DimuroIPMU}
Dimuro GP, Bedregal B, Bustince H, Mesiar R, Asiain MJ (2014{\natexlab{a}}) On
  additive generators of grouping functions. In: Laurent A, Strauss O,
  Bouchon-Meunier B, Yager RR (eds) Information Processing and Management of
  Uncertainty in Knowledge-Based Systems, Communications in Computer and
  Information Science, vol 444, Springer International Publishing, pp 252--261

\bibitem[{Dimuro et~al.(2014{\natexlab{b}})Dimuro, Bedregal, and
  Santiago}]{DBS14}
Dimuro GP, Bedregal B, Santiago RHN (2014{\natexlab{b}}) On
  {$(G,N)$}-implications derived from grouping functions. Information Sciences
  279:1 -- 17

\bibitem[{Dimuro et~al.(2016)Dimuro, Bedregal, Bustince, Asi\'ain, and
  Mesiar}]{additive-generators-FSS}
Dimuro GP, Bedregal B, Bustince H, Asi\'ain MJ, Mesiar R (2016) On additive
  generators of overlap functions. Fuzzy Sets and Systems 287:76 -- 96, theme:
  Aggregation Operations

\bibitem[{Dimuro et~al.(2017)Dimuro, Bedregal, Bustince, Jurio, Baczy{\' n}ski,
  and Mi{\' s}}]{Dimuro2017-ql}
Dimuro GP, Bedregal B, Bustince H, Jurio A, Baczy{\' n}ski M, Mi{\' s} K (2017)
  {QL}-operations and {QL}-implication functions constructed from tuples
  {$(O,G,N)$} and the generation of fuzzy subsethood and entropy measures.
  International Journal of Approximate Reasoning 82:170 -- 192

\bibitem[{Dimuro et~al.(2019{\natexlab{a}})Dimuro, Bedregal, Fernandez,
  Sesma-Sara, Pintor, and Bustince}]{dimuro2019law}
Dimuro GP, Bedregal B, Fernandez J, Sesma-Sara M, Pintor JM, Bustince H
  (2019{\natexlab{a}}) The law of {O}-conditionality for fuzzy implications
  constructed from overlap and grouping functions. International Journal of
  Approximate Reasoning 105:27--48

\bibitem[{Dimuro et~al.(2019{\natexlab{b}})Dimuro, Santos, Bedregal, Borges,
  Palmeira, Fernandez, and Bustince}]{d-impl}
Dimuro GP, Santos H, Bedregal B, Borges EN, Palmeira E, Fernandez J, Bustince H
  (2019{\natexlab{b}}) On {D}-implications derived by grouping functions. In:
  {FUZZ-IEEE} 2019, {IEEE} International Conference on Fuzzy Systems,
  Proceedings, IEEE, Los Alamitos, pp 61--66

\bibitem[{Dolati et~al.(2013)Dolati, S\'{a}nchez, and \'{U}beda
  Flores}]{Dolati13}
Dolati A, S\'{a}nchez JF, \'{U}beda Flores M (2013) A copula-based family of
  fuzzy implication operators. Fuzzy Sets and Systems 211:55 -- 61, theme:
  Aggregation operators

\bibitem[{Dubois and Prade(1984)}]{DP84}
Dubois D, Prade H (1984) A theorem on implication functions defined from
  triangular norms. Stochastica 8(3):267--279

\bibitem[{Fodor and Roubens(1994)}]{fodor1994}
Fodor J, Roubens M (1994) Fuzzy preference modelling and multicriteria decision
  support. Kluwer, New York

\bibitem[{G\'omez et~al.(2016)G\'omez, Rodr\'iguez, Montero, Bustince, and
  Barrenechea}]{Gomez201657}
G\'omez D, Rodr\'iguez JT, Montero J, Bustince H, Barrenechea E (2016)
  n-{D}imensional overlap functions. Fuzzy Sets and Systems 287:57 -- 75,
  theme: Aggregation Operations

\bibitem[{Jayaram(2008)}]{jay08}
Jayaram B (2008) On the law of importation $(x \wedge y) \rightarrow z \equiv
  (x \rightarrow (y \rightarrow z))$ in fuzzy logic. IEEE Transactions on Fuzzy
  Systems 16(1):130--144

\bibitem[{Jurio et~al.(2013)Jurio, Bustince, Pagola, Pradera, and
  Yager}]{Jurio201369}
Jurio A, Bustince H, Pagola M, Pradera A, Yager R (2013) Some properties of
  overlap and grouping functions and their application to image thresholding.
  Fuzzy Sets and Syst 229:69--90

\bibitem[{Klement et~al.(2000)Klement, Mesiar, and Pap}]{KMP00}
Klement E, Mesiar R, Pap E (2000) Triangular Norms, Trends in Logic -- Studia
  Logica Library, vol~8. Kluwer Academic Publishers, Dordrecht

\bibitem[{Klir and Yuan(1995)}]{Klir95}
Klir G, Yuan B (1995) Fuzzy Sets and Fuzzy Logic: Theory and Applications.
  Prentice Hall PTR

\bibitem[{Liu(2011)}]{Liu2011783}
Liu HW (2011) Two classes of pseudo-triangular norms and fuzzy implications.
  Computers \& Mathematics with Applications 61(4):783 -- 789

\bibitem[{Liu(2012)}]{LIU12}
Liu HW (2012) Semi-uninorms and implications on a complete lattice. Fuzzy Sets
  and Systems 191:72 -- 82, theme: Aggregation

\bibitem[{Mas et~al.(2007{\natexlab{a}})Mas, Monserrat, and Torrens}]{Mas07}
Mas M, Monserrat M, Torrens J (2007{\natexlab{a}}) Two types of implications
  derived from uninorms. Fuzzy Sets and Systems 158(3):2612--2626

\bibitem[{Mas et~al.(2007{\natexlab{b}})Mas, Monserrat, Torrens, and
  Trillas}]{Mas07a}
Mas M, Monserrat M, Torrens J, Trillas E (2007{\natexlab{b}}) A survey on fuzzy
  implication functions. IEEE Transactions on Fuzzy Systems 15(6):1107--1121

\bibitem[{Mendelson(2015)}]{Mendelson15}
Mendelson E (2015) Introduction to Mathematical Logic, 6th edn. Discrete
  Mathematics and Its Applications, Chapman and Hall/CRC, Boca Raton, FL, USA

\bibitem[{Pinheiro et~al.(2017)Pinheiro, Bedregal, Santiago, and
  Santos}]{PBSS17}
Pinheiro J, Bedregal B, Santiago RHN, Santos H (2017) {(T,N)}-implications. In:
  2017 IEEE International Conference on Fuzzy Systems (FUZZ-IEEE), Naples, pp
  1--6

\bibitem[{Pinheiro et~al.(2018{\natexlab{a}})Pinheiro, Bedregal, Santiago,
  Santos, and Dimuro}]{Vania-Nafips18}
Pinheiro J, Bedregal B, Santiago R, Santos H, Dimuro GP (2018{\natexlab{a}})
  {(T,N)}-implications and some functional equations. In: Barreto GA, Coelho R
  (eds) Fuzzy Information Processing, Springer International Publishing, Cham,
  pp 302--313

\bibitem[{Pinheiro et~al.(2018{\natexlab{b}})Pinheiro, Bedregal, Santiago, and
  Santos}]{PBSS18}
Pinheiro J, Bedregal B, Santiago RH, Santos H (2018{\natexlab{b}}) A study of
  {(T,N)}-implications and its use to construct a new class of fuzzy subsethood
  measure. International Journal of Approximate Reasoning 97:1 -- 16

\bibitem[{Pinheiro et~al.(2018{\natexlab{c}})Pinheiro, Bedregal, Santiago, and
  Santos}]{Vania-Fuzz18}
Pinheiro J, Bedregal B, Santiago RHN, Santos H (2018{\natexlab{c}})
  {$(N',T,N)$}-implications. In: 2018 IEEE International Conference on Fuzzy
  Systems (FUZZ-IEEE), Rio de Janeiro, pp 1--6

\bibitem[{Pinheiro et~al.(2018{\natexlab{d}})Pinheiro, Bedregal, Santiago, and
  Santos}]{PBSS18b}
Pinheiro J, Bedregal BRC, Santiago RHN, Santos HS (2018{\natexlab{d}}) Crisp
  fuzzy implications. In: Fuzzy Information Processing - 37th Conf. of the
  North American Fuzzy Information Processing Society, {NAFIPS} 2018,
  Fortaleza, Brazil, 2018, Proceedings, pp 348--360

\bibitem[{Pradera et~al.(2016)Pradera, Beliakov, Bustince, and {De
  Baets}}]{Pradera16}
Pradera A, Beliakov G, Bustince H, {De Baets} B (2016) A review of the
  relationships between implication, negation and aggregation functions from
  the point of view of material implication. Information Sciences 329:357 --
  380, special issue on Discovery Science

\bibitem[{Qiao and Hu(2018{\natexlab{a}})}]{QIAO2018107}
Qiao J, Hu BQ (2018{\natexlab{a}}) The distributive laws of fuzzy implications
  over overlap and grouping functions. Information Sciences 438:107 -- 126

\bibitem[{Qiao and Hu(2018{\natexlab{b}})}]{QIAO20181}
Qiao J, Hu BQ (2018{\natexlab{b}}) On multiplicative generators of overlap and
  grouping functions. Fuzzy Sets and Systems 332:1 -- 24, theme: Aggregation
  and Operators

\bibitem[{{Qiao} and {Hu}(2018)}]{8118195}
{Qiao} J, {Hu} BQ (2018) On the distributive laws of fuzzy implication
  functions over additively generated overlap and grouping functions. IEEE
  Transactions on Fuzzy Systems 26(4):2421--2433

\bibitem[{Qiao and Hu(2018)}]{QIAO20181a}
Qiao J, Hu BQ (2018) On the migrativity of uninorms and nullnorms over overlap
  and grouping functions. Fuzzy Sets and Systems 346:1 -- 54

\bibitem[{Qiao and Hu(2019{\natexlab{a}})}]{QIAO2018}
Qiao J, Hu BQ (2019{\natexlab{a}}) On generalized migrativity property for
  overlap functions. Fuzzy Sets and Systems 357:91--116

\bibitem[{Qiao and Hu(2019{\natexlab{b}})}]{QIAO201958}
Qiao J, Hu BQ (2019{\natexlab{b}}) On homogeneous, quasi-homogeneous and
  pseudo-homogeneous overlap and grouping functions. Fuzzy Sets and Systems
  357:58 -- 90, theme: Aggregation Functions

\bibitem[{Reiser et~al.(2013)Reiser, Bedregal, and Baczy\'nski}]{RBB13}
Reiser RHS, Bedregal BRC, Baczy\'nski M (2013) Aggregating fuzzy implications.
  Information Sciences 253:126--146

\bibitem[{Santos et~al.(2019)Santos, Couso, Bedregal, Tak{\'{a}}c,
  Min{\'{a}}rov{\'{a}}, Asiain, Barrenechea, and Bustince}]{Helida2019}
Santos HS, Couso I, Bedregal BRC, Tak{\'{a}}c Z, Min{\'{a}}rov{\'{a}} M, Asiain
  A, Barrenechea E, Bustince H (2019) Similarity measures, penalty functions,
  and fuzzy entropy from new fuzzy subsethood measures. International Journal
  of Intelligent Systems 34(6):1281--1302

\bibitem[{Trillas(1979)}]{Tri79}
Trillas E (1979) On negation functions in the theory of fuzzy sets. Stochastica
  3(1):47--59

\bibitem[{\v{S}t\v{e}pni\v{c}ka and De~Baets(2013)}]{SdB13}
\v{S}t\v{e}pni\v{c}ka M, De~Baets B (2013) Implication-based models of monotone
  fuzzy rule bases. Fuzzy Sets and Systems 232:134 -- 155, fuzzy Set Theory and
  Applications Selected papers from FSTA 2012

\bibitem[{Wang(2006)}]{WANG06}
Wang Z (2006) Generating pseudo-t-norms and implication operators. Fuzzy Sets
  and Systems 157(3):398 -- 410

\bibitem[{Xie et~al.(2012)Xie, Liu, Qin, and Zeng}]{Xie2012209}
Xie A, Liu H, Qin F, Zeng Z (2012) Solutions to the functional equation
  {$I(x,y)=I(x,I(x,y))$} for three types of fuzzy implications derived from
  uninorms. Information Sciences 186(1):209 -- 221

\bibitem[{Yager(2004)}]{Yager04}
Yager RR (2004) On some new classes of implication operators and their role in
  approximate reasoning. Information Sciences 167(1):193 -- 216

\end{thebibliography}
\bibliographystyle{plain}      

%
%
%

\end{document}